\documentclass[11pt]{amsart}
\usepackage {enumerate}
\usepackage{setspace}
\usepackage{caption}
\usepackage{mathrsfs}
\usepackage{hyperref}
\usepackage{esint}
\usepackage{amssymb}
\usepackage{amsmath}
\usepackage{stmaryrd}
\usepackage{graphicx}
\usepackage{enumitem}
\usepackage{comment}
\usepackage{xcolor}
\usepackage{lipsum}

\usepackage{geometry}    
\newtheorem{theorem}{Theorem}[section]
\newtheorem{lemma}[theorem]{Lemma}
\newtheorem{question}[theorem]{Question}
\newtheorem{corollary}[theorem]{Corollary}
\newtheorem{proposition}[theorem]{Proposition}

\numberwithin{equation}{section}

\theoremstyle {definition}
\newtheorem{definition}[theorem]{Definition}
\newtheorem{remark}[theorem]{Remark}

\DeclareMathOperator{\Ric}{Ric}
\DeclareMathOperator{\Rm}{Rm}
\DeclareMathOperator{\tr}{tr}
\DeclareMathOperator{\Div}{div}

\DeclareMathOperator{\diag}{diag}

\DeclareMathOperator{\Span}{span}
\DeclareMathOperator{\secf}{\mathrm {I\!I}}

\begin{document}
\title[Metric extension problem]{The metric extension problem under positive and negative curvature conditions}

\author{Wenlong Wang}
\address{School of Mathematical Sciences and LPMC, Nankai University, Tianjin, 300071, People's Republic of China}
\email{wangwl@nankai.edu.cn}
\author{Jintian Zhu}
\address{Institute for Theoretical Sciences, Westlake University, 600 Dunyu Road, 310030, Hangzhou, Zhejiang, People's Republic of China}
\email{zhujintian@westlake.edu.cn}

\begin{abstract}
We first prove that every smooth boundary metric on a compact manifold extends to a metric with any prescribed positive lower bound for Ricci curvature, whereas extensions under stronger positive \(k^{\mathrm{th}}\)-intermediate Ricci curvature conditions may fail due to local obstructions. For negative sectional curvature, we identify a global obstruction to extension. We then consider the class of compact manifolds defined by the existence of a metric with negative sectional curvature and boundary index at most one. On every manifold in this class, we prove that any smooth boundary metric extends to a metric with negative sectional curvature and strictly convex umbilical boundary. We further show that every such initial metric admits a complete asymptotically hyperbolic isometric extension with the same curvature condition and any prescribed conformal infinity. This class of manifolds is closed under boundary connected sums. As a geometric application of these extension results, every smooth metric on \(\mathbb S^n\) admits a strictly convex isometric embedding into \(\mathbb R^{n+1}\) equipped with a complete metric of negative sectional curvature. The proofs combine neck constructions with corner smoothing for upper curvature bounds.

\end{abstract}

\subjclass[2020]{Primary 53C21; Secondary 53C20}

\maketitle

\tableofcontents

\section{Introduction}

A central problem in Riemannian geometry is to construct metrics on a manifold with prescribed curvature properties, or to identify topological obstructions to their existence. For complete Riemannian manifolds without boundary, global curvature conditions often impose strong geometric and topological restrictions, as illustrated by the Bonnet--Myers, splitting, soul, and Cartan--Hadamard theorems.

In sharp contrast, for compact manifolds with boundary, when no boundary conditions are imposed, the situation is remarkably flexible. By a result of Gromov \cite{Gro1969} based on his $h$-principle, every open $n$-manifold $(n\geq 2)$ admits (possibly incomplete) Riemannian metrics of positive as well as metrics of negative sectional curvature. Streil \cite{Streil2017} strengthened this result by proving the existence of metrics with sectional curvature in any prescribed open interval.

This flexibility suggests imposing boundary conditions on compact manifolds with boundary. A natural choice is to prescribe the boundary metric together with either the full second fundamental form or only part of it, such as the mean curvature. Finding Riemannian metrics on a fixed compact manifold with boundary that realize the prescribed boundary data and satisfy a global curvature condition is referred to as the \emph{extension problem}. A less restrictive variant is the \emph{fill-in problem}, in which the boundary data are fixed while the interior topology is allowed to vary; see \cite[Section 5$\frac{5}{7}$]{Gro1996} for motivation.

In recent years, Gromov has highlighted the extension and fill-in problems in the setting of nonnegative scalar curvature; see \cite{Gro2019} and \cite[Section 3.12]{Gro2023}. In addition to their intrinsic geometric interest, these problems are closely intertwined with quasi-local mass and positive mass theorems; see, for example, \cite{BrendleH,CLSZ,FHH,J2013,MM2017,ST2002,ST2007,SWW2022,Wang Y}.

Against this background, it is natural to begin with the following basic {\it metric extension problem}:
\begin{question}\label{metric extension problem}
Let $M$ be a compact manifold with boundary, and let $\gamma$ be a smooth Riemannian metric on $\partial M$. Does $\gamma$ always extend to a smooth Riemannian metric $g$ on $M$ satisfying a prescribed curvature condition? 
\end{question}

This question may be viewed as asking to what extent the \(h\)-principle-type flexibility exhibited by curvature conditions on open manifolds persists when the boundary metric is prescribed. In the case of positive scalar curvature, the problem was posed by Gromov and answered affirmatively by Shi, Wei, and the first-named author \cite{SWW2022}:

\begin{theorem}[Shi--Wang--Wei \cite{SWW2022}]\label{PSCextension}
Let $M$ be a compact manifold with boundary. Then every smooth Riemannian metric $\gamma$ on $\partial M$ can be extended to a smooth Riemannian metric $g$ on $M$ with positive scalar curvature. 
\end{theorem}

Using this theorem, Miao \cite{Miao2021} ruled out nonnegative scalar curvature fill-ins with sufficiently large boundary mean curvature, thereby resolving the non-fill-in problem formulated by Gromov (see \cite[Question $ \mathrm A_1$]{Gro2019}). Gromov further strengthened the extension theorem by allowing any prescribed lower bound on scalar curvature and extending the result to manifolds with corners (see \cite[Section 3.12]{Gro2023}).

Motivated by the extension theorem for positive scalar curvature and its applications, we study the metric extension problem for the following curvature conditions:
\begin{itemize}[left=0.5cm]
    \item positive Ricci curvature;
    \item positive $k^{\rm th}$-intermediate Ricci curvature;
    \item negative sectional curvature and negative curvature operator.
\end{itemize}

\subsection{Metric extension problem for positive curvature conditions}
Our first result shows that every boundary metric admits an extension with positive Ricci curvature. In fact, one can even prescribe an arbitrary positive lower bound:
\begin{theorem}\label{Thm:main 1}
Let \(M^n\) be a compact manifold with boundary, where \(n\geq2\). Given any \(K>0\), every smooth Riemannian metric \(\gamma\) on \(\partial M\) can be extended to a smooth Riemannian metric \(g\) on \(M\) such that \(\Ric^g>Kg\).
\end{theorem}

In general, one cannot further require $\partial M$ to be strictly convex. For instance, a compact $3$-manifold with nonnegative Ricci curvature and strictly convex boundary must be diffeomorphic to the $3$-ball (see \cite{FL}).

The passage from positive scalar to positive Ricci curvature naturally raises the question whether similar extension results hold for stronger positive curvature conditions. However, our results show that, within the hierarchy of \(k^{\rm th}\)-intermediate Ricci curvatures, positive Ricci curvature is in general the strongest condition for which such an extension theorem can hold. We recall the definition of \(k^{\rm th}\)-intermediate Ricci curvature:

\begin{definition}\label{k-Ricci}
Let $(M^n,g)$ be a Riemannian manifold. For a point $p\in M$ and an orthonormal set $\{v_0;v_1,\ldots,v_k\}\subset T_pM$, the $k^{\rm th}$-intermediate Ricci curvature of $g$ at $p$ is defined by
$$\Ric^g_k(v_0;v_1,\ldots,v_k)=\sum_{i=1}^k\sec^g(v_0\wedge v_i),$$
where $\sec^g$ denotes the sectional curvature of $g$. For $\sigma\in\mathbb R$, we write $\Ric^g_k>\sigma$ if $\Ric^g_k(v_0;v_1,\ldots,v_k)>\sigma$ for all $p\in M$ and all orthonormal sets $\{v_0;v_1,\ldots,v_k\}\subset T_pM$. The condition $\Ric^g_k<\sigma$ is defined analogously. 
\end{definition}

The intermediate Ricci curvatures interpolate between sectional and Ricci curvature: $
\Ric^g_1\left(v_0;v_1\right)=\sec^g\left(v_0\wedge v_1\right)$, $\Ric^g_{n-1}\left(v_0;v_1,\ldots,v_{n-1}\right)=\Ric^g\left(v_0,v_0\right)$.
Moreover, $\Ric^g_k>0$ implies $\Ric^g_l>0$ for all $l \geq k$. For a comprehensive list of references, see Mouill\'e’s online bibliography \cite{Mou2022}.

Beyond the Ricci case, extensions with $\Ric_k>0$ can fail due to purely {\it local} obstructions, such as insufficient positivity of the boundary curvature. This is made precise in the following theorem: 

\begin{theorem}\label{Thm: obstruction Ric_k}
Let $M$ be an $n$-dimensional compact manifold with boundary, and let $\sigma$ be a constant. Then the following hold:
\begin{itemize}[left=0.5cm]
\item[(i)]  For $n\geq 4$ and $k\leq n-3$, if $\gamma$ is a smooth metric on $\partial M$ with $\Ric_k^\gamma\leq \sigma$ at some point, then there is no smooth metric $g$ on any collar neighborhood of $\partial M$ such that $\Ric_k^g>\sigma$ and $g|_{\partial M}=\gamma$.
\item[(ii)] For $n\geq 3$, if $\gamma$ is a smooth metric on $\partial M$ with $\Ric^\gamma\leq \sigma$, and there exists a smooth metric $g$ on a collar neighborhood of $\partial M$ such that $\Ric_{n-2}^g>\sigma$ and $g|_{\partial M}=\gamma$, then the boundary tangent bundle $T(\partial M)$ splits as a direct sum of two nontrivial subbundles.
\end{itemize}
\end{theorem}

Applying Theorem \ref{Thm: obstruction Ric_k} to suitable boundary metrics, we obtain the following concrete non-extension results:
\begin{corollary}\label{Cor: even sphere}
The following hold:
\begin{itemize}[left=0.5cm]
\item[(i)] For any compact $n$-manifold $M$ with boundary, $n\geq 4$, there exists a smooth metric $\gamma$ on $\partial M$ that cannot be extended to a smooth metric on $M$ with positive $(n-3)^{\rm th}$-intermediate Ricci curvature.
\item[(ii)] For every odd $n\geq 5$, there exists a smooth metric $\gamma$ on $\mathbb S^{n-1}$ that cannot be extended to a smooth metric on $\mathbb D^n$ with positive $(n-2)^{\rm th}$-intermediate Ricci curvature.
\end{itemize}
\end{corollary}

\subsection{Metric extension problem for negative curvature conditions}
By a result of Gromov \cite{Gro1969}, every compact $n$-manifold with boundary, $n\geq 2$, admits a metric with negative sectional curvature. For Ricci curvature, Gao and Yau \cite{GY} proved that every $3$-manifold of finite topological type admits a complete metric with negative Ricci curvature. Lohkamp \cite{Loh1994} extended this result to all manifolds without boundary of dimension \(n\geq 3\). He also proved that any metric on the boundary of a compact \(n\)-manifold, \(n\geq 3\), extends to a metric with negative Ricci curvature for which the boundary is totally geodesic. This naturally raises the question of whether prescribing the boundary metric introduces any obstruction to an extension with negative sectional curvature.

Beyond negative sectional curvature, we also consider negative curvature operator, which may impose stronger topological constraints. For instance, Aravinda and Farrell \cite{AF05} showed that there exist closed manifolds that admit metrics with negative sectional curvature but no metric with nonpositive curvature operator.  

As for positive sectional curvature, the metric extension problem for negative sectional curvature is not always solvable.

\begin{proposition}\label{Prop: global obstruction}
Let \(M=(\mathbb S^1\times\mathbb D^2)\#(\mathbb S^2\times\mathbb S^1)\),
where the connected sum is taken in the interior. Then no flat metric on
\(\partial M\cong\mathbb T^2\) can be extended to a smooth metric on \(M\)
with nonpositive sectional curvature.
\end{proposition}

In contrast to the positive sectional curvature case, the obstruction here is \emph{global}. Indeed, for any smooth metric $\gamma$ on $\partial M$, one can always extend $\gamma$ to a smooth metric with negative curvature operator on a small collar neighborhood of $\partial M$ (see Lemma \ref{Lem: NSecCcobordismlm}); however, the topology of $M$ may obstruct a global continuation of the local extension. Nevertheless, this obstruction vanishes for certain manifolds.

We use the following convention and terminology. For a hypersurface with a chosen unit normal $\nu$, its second fundamental form is defined by
\[
\mathrm{II}(X,Y)
=
\langle \nabla_X Y,\nu\rangle
\]
The index of the boundary at a point is defined as the negative index of inertia of the second fundamental form with respect to the inward unit normal. In particular, a convex boundary has index \(0\) at every point. Let $\mathcal C_n$ denote the class of compact $n$-manifolds with boundary that admit a smooth metric with negative sectional curvature such that the boundary has index at most $1$ everywhere. Define \(\mathcal C_n'\) analogously, with negative curvature operator in place of negative sectional curvature. Clearly, \(\mathcal C_n'\subset \mathcal C_n\).

\begin{theorem}\label{Thm:main 2}
Let $M\in \mathcal C_n$ (respectively, $\mathcal C'_n$). Then any smooth metric $\gamma$ on $\partial M$ can be extended to a smooth metric on $M$ with negative sectional curvature (respectively, negative curvature operator) and strictly convex umbilical boundary.
\end{theorem}

\begin{remark}
The index bound in Theorem \ref{Thm:main 2} is sharp: it cannot in general be relaxed to $2$. Indeed, let
\(
M=(\mathbb S^1\times\mathbb D^2)\#(\mathbb S^2\times\mathbb S^1).
\)
By Gromov's \(h\)-principle, \(M\) admits a smooth metric with negative sectional curvature. Since \(\partial M\) is two-dimensional, its index is at most \(2\). However, Proposition \ref{Prop: global obstruction} shows that a flat metric on \(\partial M\) cannot be extended to a metric on \(M\) with even nonpositive sectional curvature. 
\end{remark}

A basic class of examples in $\mathcal C'_n$ is given by compact cores of convex cocompact hyperbolic $n$-manifolds. Further examples are provided by the following structural properties of $\mathcal C_n$ and $\mathcal C'_n$.

\begin{theorem}\label{Thm: structure}
The following statements hold:
\begin{itemize}[left=0.5cm]
\item[(i)] The classes $\mathcal C_n$ and $\mathcal C'_n$ are closed under boundary connected sums and attachments of boundary $1$-handles. In particular, $\mathcal C'_n$ contains all $n$-dimensional $1$-handlebodies. 
\item[(ii)] Let $N^k$ be a closed manifold that admits a metric with negative sectional curvature (respectively, negative curvature operator), and let $E$ be a rank $m$ vector bundle over $N$. Then the disk bundle \(D(E)\) belongs to \(\mathcal C_{k+m}\) (respectively, \(\mathcal C'_{k+m}\)).
\end{itemize}
\end{theorem}

In dimension $3$, the situation is much clearer. Since every bivector in dimension $3$ is decomposable, $\mathcal C_3=\mathcal C'_3$. Combining Theorem \ref{Thm:main 2} with Hass's characterization of compact orientable
$3$-manifolds admitting a metric of negative sectional curvature with convex boundary \cite[Section~3.4]{Hass}, we obtain that, for a compact orientable $3$-manifold $M$ with nonempty boundary, $M\in\mathcal C_3$ if and only if
$M$ is irreducible and contains no embedded $\pi_1$-injective torus. 

For such $M$, every smooth metric $\gamma$ on $\partial M$ can in fact be extended to a metric of constant negative sectional curvature with strictly convex boundary. After a suitable rescaling, this follows from the classical theorem of Aleksandrov and Pogorelov \cite{Alex,Pog} when $M$ is a $3$-ball, from Labourie's theorem \cite{Labourie} when $M$ is neither a $3$-ball nor a solid torus, and from Schlenker's treatment \cite{Schlenker} of the remaining solid-torus case.

For a compact Riemannian manifold \((M,g)\) with boundary and negative sectional curvature, it is also natural to ask whether \((M,g)\) admits an isometric extension to a complete open Riemannian manifold of negative sectional curvature that is diffeomorphic to its interior \(\mathring M\). Pigola and Veronelli proved that such an extension exists when $\partial M$ is strictly convex; see \cite[Theorem M]{PV}. Our result relaxes the strict convexity assumption to the requirement that the boundary have index at most one everywhere. In addition, the extension can be chosen to be asymptotically hyperbolic with any prescribed conformal infinity.

\begin{theorem}\label{Thm:conformally compact}
Let \(M\) be a compact manifold with boundary, and let \(g\) be a smooth metric on \(M\) with
negative sectional curvature such that the boundary has index at most \(1\)
everywhere. Then, for every smooth Riemannian metric \(\gamma\) on
\(\partial M\), there exist a complete asymptotically hyperbolic metric
\(\tilde g\) on \(\mathring M\) with negative sectional curvature and
conformal infinity \((\partial M,[\gamma])\), and an isometric embedding
\[
X\colon (M,g)\hookrightarrow (\mathring M,\tilde g).
\]
The same conclusion holds with negative sectional curvature replaced by negative curvature operator throughout.
\end{theorem}

Here, \emph{asymptotically hyperbolic} (AH) means conformally compact with sectional curvatures tending to $-1$ at infinity; see Subsection \ref{Prescribing conformal infinity} for details. This notion is sometimes called \emph{weakly asymptotically hyperbolic}, to distinguish it from stronger notions of hyperbolic asymptotics; see, for example, \cite{Wang X}.

Interestingly, Theorem \(\ref{Thm:conformally compact}\) also plays a key role in the proof of Theorem \(\ref{Thm: structure}\) (i). By prescribing conformal infinities represented by metrics that are flat near selected points, a slight refinement of the construction makes the corresponding AH extensions exactly hyperbolic near those points. This permits a Maskit-type gluing at conformal infinity; cf. \cite{ChruscielDelay,ILS,MazzeoPacard}. By attaching boundary \(1\)-handles in this way, we obtain complete negatively curved isometric extensions with noncylindrical added regions. This answers Pigola and Veronelli's question \cite[Remark~6.6]{PV} of whether the added region must have cylindrical topology.

Theorems \ref{Thm: structure} (ii) and \ref{Thm:conformally compact} show that every vector bundle over a closed manifold carrying a metric with negative sectional curvature (respectively, negative curvature operator) admits a complete AH metric with the same curvature condition. Anderson \cite{Anderson} previously proved the existence of complete metrics of negative sectional curvature on such bundles. The proof of Theorem \ref{Thm: structure} (ii) requires only a local construction near the zero section, which Theorem \ref{Thm:conformally compact} turns into a global extension, illustrating the local-to-global nature of our method. 

Our final application of the extension method concerns an isometric realization problem. The classical prototype is the Weyl problem: Nirenberg proved that every smooth metric of positive Gauss curvature on $\mathbb S^2$ admits a strictly convex isometric embedding into $\mathbb R^3$ \cite{Nirenberg}. The hyperbolic version, due to Aleksandrov and Pogorelov \cite{Alex,Pog}, gives such an embedding into $\mathbb H^3$ for every smooth metric on $\mathbb S^2$ with Gauss curvature greater than $-1$. For a torus or a higher-genus surface, the hyperbolic extension results stated above \cite{Labourie,Schlenker}, followed by the attachment of the canonical hyperbolic end, yield a strictly convex isometric embedding into a complete handlebody interior of the same genus with constant negative sectional curvature.

We prove a higher-dimensional analogue in which the ambient metric is allowed to have variable negative curvature. Combining the preceding metric-extension theorem with the AH extension theorem gives the following.

\begin{theorem}\label{Thm: realization} Let \(M\in\mathcal C_n\) (respectively, \(\mathcal C'_n\)). For every smooth metric \(\gamma\) on \(\partial M\), there exist an asymptotically hyperbolic metric \(g\) on \(\mathring M\) with negative sectional curvature (respectively, negative curvature operator) and a smooth isometric embedding
\(
X\colon(\partial M,\gamma)\hookrightarrow(\mathring M,g).
\)
Moreover, the second fundamental form of \(X(\partial M)\) with respect to the inward normal is positive definite.
\end{theorem}

As a concrete consequence, every smooth metric on
\(
\mathop{\#}_{j=1}^{k}
\bigl(\mathbb S^{n-1}\times\mathbb S^1\bigr)
\) admits a strictly convex isometric embedding into the interior of an \((n+1)\)-dimensional \(1\)-handlebody of genus \(k\), equipped with an AH metric with negative curvature operator. In general, the ambient manifold cannot be required to be simply connected; see Remark \(\ref{can'tsimconnec}\). For spheres, however, the ambient manifold can be chosen to be simply connected and hyperbolic outside a compact set.

\begin{theorem}\label{Thm: sphererealization}
For every smooth metric \(\gamma\) on \(\mathbb S^n\), there exist a complete metric \(g\) on \(\mathbb R^{n+1}\) with negative curvature operator and a smooth isometric embedding
\(
X\colon(\mathbb S^n,\gamma)\hookrightarrow(\mathbb R^{n+1},g).
\)
Moreover, \(g\) has constant negative sectional curvature outside a compact set, and the second fundamental form of \(X(\mathbb S^n)\) with respect to the inward normal is positive definite.
\end{theorem}

\subsection{Outline of the proof}
The proofs of the main results, Theorems \ref{Thm:main 1} and \ref{Thm:main 2}, build on the boundary-replacement argument of \cite{SWW2022}. We briefly outline the strategy. To solve the metric extension problem for a curvature condition \((\mathrm{C})\), we begin with an initial metric $g_{\rm {ini}}$ on $M$ that satisfies \((\mathrm{C})\) but imposes no {\emph{a~priori}} control on the boundary metric. 

When \((\mathrm{C})\) denotes positive Ricci curvature or negative sectional curvature, the existence of such an initial metric follows from Gromov \cite{Gro1969}; when \((\mathrm{C})\) denotes negative curvature operator, it follows from Streil’s result with pinching constant sufficiently close to \(1\) \cite{Streil2017}. However, in negative-curvature settings we further require the boundary to have index at most $1$ for the subsequent gluing-and-smoothing step; this is precisely why Theorem \ref{Thm:main 2} requires index at most $1$. 

With the initial metric in hand, the goal is to perform a boundary replacement so that \(\partial M\) ultimately carries the prescribed metric \(\gamma\) while \((\mathrm{C})\) is preserved. Following \cite{SWW2022}, we achieve this by attaching an {\it extension neck} along \(\partial M\) to \((M, g_{\mathrm{ini}})\), as illustrated in Figure \ref{Fig: replacing}.

\begin{figure}[htbp]
\centering
\includegraphics[width=6cm]{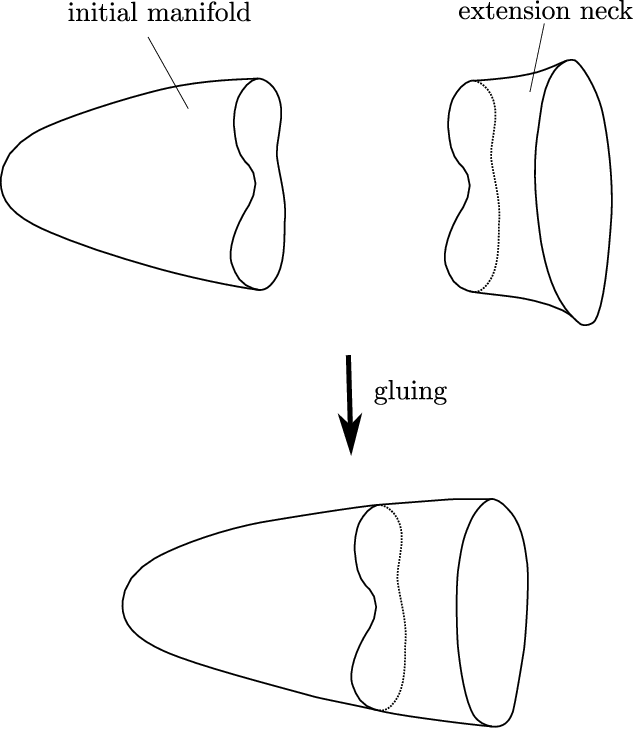}
\caption{The boundary-replacement argument}
\label{Fig: replacing}
\end{figure}

The boundary replacement argument consists of two steps. In the first step, we construct an extension neck satisfying:
\begin{itemize}[left=0.5cm]
\item [(i)]  The neck is diffeomorphic to $\partial M\times [0,1]$.
\item [(ii)] The neck satisfies the curvature condition $(\mathrm C)$.
\item [(iii)] The induced metrics on $\partial M\times\{0\}$ and $\partial M\times\{1\}$ coincide with  $g_{\rm{ini}}|_{\partial M}$ and $\gamma$, respectively. Moreover, the second fundamental form of $\partial M\times \{0\}$ satisfies additional conditions ensuring that $(\mathrm C)$ holds after gluing and smoothing.  
\end{itemize}

Our construction uses the quasi-spherical metric method. In the quasi-spherical ansatz, the curvature condition typically translates into an evolutionary partial differential inequality (PDI) formulation for the lapse function $u$, and the second fundamental form condition along $\partial M\times \{0\}$ is encoded as the initial value of $u$. Compared with the scalar curvature case, which reduces to a single PDI, our setting yields a PDI system for $u$ and is therefore substantially more complex, particularly in the negative curvature operator case.

For positive Ricci curvature, inspired by Gromov's simplification of the proof of Theorem \ref{PSCextension} (Section 3.12 of \cite{Gro2023}), we reduce the corresponding PDI system to a single ordinary differential inequality (ODI) for $u$. For negative curvature operator, control by an analogous but sign-reversed ODI is insufficient. However, the negative-curvature requirement allows us to exploit the Gauss equation for convex neck slices, yielding an additional curvature-control mechanism unavailable in positive-curvature settings. With the Gauss-equation control, we ultimately reduce the PDI system to an ODI supplemented by a smallness condition on $u$. Details are given in Lemmas \ref{Lem:PRicciCcobordismlm} and \ref{Lem: NSecCcobordismlm}, which may be of independent interest.

After gluing the desired extension neck to the initial manifold, one generally obtains a metric with corners; the second step is to smooth these corners while preserving the curvature condition $(\mathrm C)$. Perelman \cite{Per1997} resolved the positive-Ricci case under the condition that the sum of the second fundamental forms from the two sides of the gluing hypersurface is positive. 

For negative sectional curvature, however, the situation is more subtle: merely reversing the sign of this sum is insufficient to preserve negative curvature, even in the weaker Alexandrov sense (see Kosovskiĭ’s criterion \cite{Kos2}, recalled in Theorem \ref{Kos2} below). To preserve negative curvature under smoothing, we prove a smoothing theorem (Theorem \ref{thm: smoothing sec}) under a \(2\)-convexity assumption. This theorem also applies to metrics with negative curvature operator.

\subsection{Organization of the paper} In Section \ref{Sec: 2}, we construct the requisite extension necks with positive Ricci curvature and with negative curvature operator. In Section \ref{Sec: 3}, we establish a corner-smoothing theorem in negative curvature settings. In Section \ref{Sec: 4}, we prove the extension theorem for positive Ricci curvature (Theorem \ref{Thm:main 1}) and a non-extension result for stronger $k^{\rm th}$-intermediate Ricci curvatures (Theorem \ref{Thm: obstruction Ric_k}). In Section \ref{Sec: 5}, we first present a global obstruction to the metric extension problem for negative sectional curvature (Proposition \ref{Prop: global obstruction}) and prove the extension theorem for the classes \(\mathcal C_n\) and \(\mathcal C'_n\) (Theorem \ref{Thm:main 2}). We then construct AH extensions with prescribed conformal infinity and establish structural properties of these classes. Finally, we apply these results to isometric realization problems.

\subsection*{Acknowledgements}
The first-named author thanks Chengjie Yu for helpful conversations and for bringing the paper \cite{Streil2017} to his attention in 2023. The authors are grateful to Yuguang Shi for insightful discussions and constant encouragement. 

\subsection*{Development of the results and AI usage statement}

The authors began this project in May 2023. Preliminary versions of the two main results---Theorem \ref{Thm:main 1} (asserting $\Ric>0$, but not yet the stronger estimate $\Ric>Kg$ for any prescribed $K>0$) and Theorem \ref{Thm:main 2} (under a $2$-convexity hypothesis rather than the present index-at-most-one hypothesis and covering only the negative sectional curvature case)---together with their supporting lemmas, were obtained later that year. The authors subsequently strengthened these results and developed further applications. The first author has presented successive versions of these results at several conferences since 2023, including the 2024 Annual Conference of the Chinese Mathematical Society. 

All mathematical ideas, arguments, and results presented in this manuscript were developed independently by the authors, \textbf{without the use of generative AI}. Following completion of the initial draft, AI tools were used only to polish the language. The authors take full responsibility for the content of the manuscript.

\section{Neck construction}\label{Sec: 2}

Bartnik \cite{Bar1993} introduced the quasi-spherical metric method to construct metrics with prescribed scalar curvature along a foliation by spheres, hence the term ``quasi-spherical".  Shi and Tam \cite{ST2002} further developed this method in their proof of the positivity of the Brown--York mass. Building on this approach, Shi, Wei, and the first-named author \cite{SWW2022} established an extension lemma for positive scalar curvature. Inspired by Gromov's simplification of that proof (see \cite[Section 3.12]{Gro2023}), we observe that the quasi-spherical metric method can also be further developed to construct necks with positive Ricci curvature and necks with negative curvature operator, under suitable boundary conditions.

\subsection{Quasi-spherical metric}\label{subsec: quasi-spherical metric} We begin by reviewing some basic facts about quasi-spherical metrics (see \cite{Bar1993,SW,ST2002}).
\subsubsection{The set-up} 
Let $\Sigma$ be a closed $(n-1)$-dimensional manifold, and let $\mathcal C=\Sigma\times [a,b]$ be a cylinder over $\Sigma$. Assume that $\mathcal C $ is equipped with a smooth base metric $g_{\rm{base}}$ of the form
\[g_{\rm{base}}=dt^2+\gamma_t,\]
where $\{\gamma_t\}_{t\in[a,b]}$ is a smooth family of Riemannian metrics on $\Sigma$. Following Bartnik, we retain the term ``quasi-spherical” for this ansatz even when $\Sigma$ is not a sphere. By a quasi-spherical metric with respect to the base metric $g_{\rm{base}}$, we mean a deformation of $g_{\rm{base}}$ of the form
\begin{equation}\label{qsform}
g=u^2dt^2+\gamma_t,
\end{equation}
where $u$ is a smooth positive function on $\mathcal C$.
In other words, we modify the unit ``lapse function" while keeping the slice metrics $\gamma_t$ fixed.

\subsubsection{Notation and conventions}\label{subsubsec: notation and conventions}
We adopt the following notation and conventions throughout. Let $(M,g)$ be a Riemannian manifold. For its curvature tensor, we use the sign convention that for any $p\in M$ and $X,Y,Z,W\in T_pM$, $$R(X,Y)Z=\nabla_X\nabla_YZ-\nabla_Y\nabla_XZ-\nabla_{[X,Y]}Z$$ and $$R(X,Y,Z,W)=\langle R(X,Y)Z,W\rangle.$$
The Ricci curvature is obtained from the $(0,4)$-curvature tensor by contracting either the first and fourth indices or, equivalently, the second and third indices.

For symmetric bilinear forms $h$ and $k$ on a vector space $V$,
their Kulkarni--Nomizu product $h\owedge k$ is the symmetric
bilinear form on $\wedge^2V$ defined by
\[
\begin{aligned}
(h\owedge k)(u\wedge v,\,w\wedge z)
:={}&h(u,w)k(v,z)+h(v,z)k(u,w)\\
&-h(u,z)k(v,w)-h(v,w)k(u,z),
\end{aligned}
\]
for $u,v,w,z\in V$. We write
\[
h\wedge h:=\frac12\,h\owedge h,
\]
so that
\[
(h\wedge h)(u\wedge v,\,w\wedge z)
=h(u,w)h(v,z)-h(u,z)h(v,w).
\]
In particular, a Riemannian metric $g$ induces the metric
$g\wedge g$ on $\wedge^2TM$.
For convenience, we will sometimes continue to write $\langle\cdot,\cdot\rangle$ for $g\wedge g$ and $|\cdot|$ for the corresponding norm. The curvature tensor $R$ induces a symmetric bilinear form $\Rm$ on $\wedge^{2}TM$ via
\[
\Rm(X\wedge Y,\, Z\wedge W):=R(X,Y,W,Z),
\quad\  p\in M,\ X,Y,Z,W\in T_{p}M.
\]
The curvature operator $\mathcal R$ is defined by
\[
\langle \mathcal R\alpha,\, \beta\rangle=\Rm(\alpha,\beta),
\quad\ p\in M,\ \alpha,\beta\in \wedge^{2}T_{p}M.
\]

For a symmetric bilinear form $h$, we write $h>0$ if it is positive definite. For $K\in\mathbb R$, we write $\mathcal R\leq K$ if all eigenvalues of $\mathcal R$ are at most $K$ pointwise; equivalently,
\[
\Rm\le Kg\wedge g.
\]

We now turn to the quasi-spherical metric setting on $\Sigma\times[a,b]$ with metric $g$ of the form \eqref{qsform}. The same conventions also apply in Section~\ref{Sec: 3} (corner smoothing): there, $u\equiv 1$, and $\{\gamma_t\}$ is smooth on each side of the corner but only Lipschitz across it.

For any $t$-dependent function or tensor field $T$, we write $T':=\partial_t T$. A semicolon followed by an index denotes covariant differentiation. Using the product structure, we identify each slice $\Sigma_t:=\Sigma\times\{t\}$ with $\Sigma$, and freely view objects tangent to $\Sigma_t$ as objects on $\Sigma$, and conversely. Let $\{e_i\}_{i=1}^{n-1}$ be a local frame on $\Sigma$. Via the above identification, we also view it as a local frame on each slice $\Sigma_t$, so that $[e_i,\partial_t]=0$. Set $e_n=u^{-1}\partial_t$. The indices $i,j,k,\ell$ range from $1$ to $n-1$. Indices are raised with the slice metric $\gamma_t$, and $(\gamma'_t)^2$ denotes the $(0,2)$-tensor field on $\Sigma_t$ defined by 
$$(\gamma'_t)^2_{ij}=(\gamma'_t)_{ik}(\gamma'_t)^k_j.$$ Let $\secf_t$ denote the second fundamental form of $\Sigma_t$ induced by $g$ with respect to the unit normal $e_n$, namely, 
$$\secf_t(e_i,e_j)=\langle\nabla^g_{e_i}e_j,e_n\rangle.$$

\subsubsection{Curvature relations}
In this subsection, we compute the curvature of the quasi-spherical metric $g$ in terms of the family $\{\gamma_{t}\}_{t\in[a,b]}$ and the lapse function $u$. For our purposes, we assume that $u$ depends only on $t$. 

\begin{proposition}\label{Lem: curvature relations}
The following relations hold:
\begin{equation}\label{secf}
\secf_t=-\frac{1}{2u}\gamma'_t,
\end{equation}
\begin{equation}\label{tn}
R^g_{nijn}=\frac{u'}{2u^3}\left(\gamma'_t\right)_{ij}+\frac{1}{4u^2}\left(\gamma'_t\right)^2_{ij}-\frac{1}{2u^2}\left(\gamma''_t\right)_{ij},
\end{equation}
\begin{equation}\label{tt}
R^g_{ijkl}=R^{\gamma_t}_{ijkl}-\frac{1}{4u^2}\left[\left(\gamma'_t\right)_{il}\left(\gamma'_t\right)_{jk}-\left(\gamma'_t\right)_{ik}\left(\gamma'_t\right)_{jl}\right],
\end{equation}
\begin{equation}\label{mixed}
R^g_{ijkn}=\frac{1}{2u}\left[\left(\nabla^{\gamma_t}\gamma'_t\right)_{ik;j}-\left(\nabla^{\gamma_t}\gamma'_t\right)_{jk;i}\right].
\end{equation}
\end{proposition}

\begin{proof}
In our setting,
$$\left\langle e_i,e_j\right\rangle=\left(\gamma_t\right)_{ij},\quad \left\langle e_i,\partial_t\right\rangle=0,\quad \left\langle \partial_t,\partial_t\right\rangle=u^2.$$
By the Koszul formula, we have
\begin{equation*}
\left\langle\nabla^g_{e_i}\partial_t,e_j\right\rangle=\frac{1}{2}\left(\gamma'_t\right)_{ij},\quad \ \left\langle\nabla^g_{e_i}\partial_t,\partial_t\right\rangle=0,
\end{equation*}
and 
\begin{equation*}
\left\langle\nabla^g_{\partial_t}\partial_t,\partial_t\right\rangle=uu',\quad \  \left\langle\nabla^g_{\partial_t}\partial_t,e_i\right\rangle=0.
\end{equation*}
Therefore,  
\begin{equation*}
\nabla^g_{e_i}\partial_t=\frac{1}{2}\left(\gamma_t\right)^{kl}\left(\gamma'_t\right)_{ik}e_l,\quad\ 
\nabla^g_{\partial_t}\partial_t=\frac{u'}{u}\partial_t.
\end{equation*}
Recalling that $e_n=u^{-1}\partial_t$, we obtain
\begin{equation*}
\secf_t(e_i,e_j)=\frac{1}{u}\langle\nabla^g_{e_i}e_j,\partial_t\rangle=-\frac{1}{u}\langle\nabla^g_{e_i}\partial_t,e_j\rangle=-\frac{1}{2u}\gamma'_t(e_i,e_j).
\end{equation*}
Differentiating $\nabla^g_{e_i}\partial_t$ once more in the $t$-direction, we compute
\begin{equation}\label{sec23}
\begin{split}
\left\langle\nabla^g_{\partial_t}\nabla^g_{e_i}\partial_t,e_j\right\rangle&=\partial_t\left\langle\nabla^g_{e_i}\partial_t,e_j\right\rangle-\left\langle\nabla^g_{e_i}\partial_t,\nabla^g_{\partial_t}e_j\right\rangle\\
&=\frac{1}{2}\left(\gamma''_t\right)_{ij}-\frac{1}{4}\left(\gamma'_t\right)^2_{ij},
\end{split}
\end{equation}
and
\begin{equation}\label{sec24}
\left\langle\nabla^g_{e_i}\nabla^g_{\partial_t}\partial_t,e_j\right\rangle=\frac{u'}{u}\left\langle\nabla^g_{e_i}{\partial_t},e_j\right\rangle=\frac{u'}{2u}\left(\gamma'_t\right)_{ij}.
\end{equation}
Combining \eqref{sec23} and \eqref{sec24}, and noting that $[e_i,\partial_t]=0$, we get
\begin{equation*}
R^g_{nijn}=\frac{u'}{2u^3}\left(\gamma'_t\right)_{ij}-\frac{1}{2u^2}\left(\gamma''_t\right)_{ij}
+\frac{1}{4u^2}\left(\gamma'_t\right)^2_{ij}.
\end{equation*}
By the Gauss equation,
\begin{equation*}
R^g_{ijkl}=R^{\gamma_t}_{ijkl}-\left(\secf_t\right)_{il}\left(\secf_t\right)_{jk}+\left(\secf_t\right)_{ik}\left(\secf_t\right)_{jl}.
\end{equation*}
By the Codazzi equation,
\begin{equation*}
R^g_{ijkn}=\left(\nabla^{\gamma_t}\secf_t\right)_{jk;i}-\left(\nabla^{\gamma_t}\secf_t\right)_{ik;j}.
\end{equation*}
Finally, substituting the relation \eqref{secf} into the Gauss and Codazzi equations, we obtain \eqref{tt} and \eqref{mixed}.
This completes the proof.
\end{proof}

\subsection{Neck construction} In this subsection, we use quasi-spherical metrics to construct extension necks in two settings: positive Ricci curvature and negative curvature operator.

\subsubsection{Necks with positive Ricci curvature}
\begin{lemma}\label{Lem:PRicciCcobordismlm}
Let $\Sigma$ be a closed $(n-1)$-dimensional manifold, and let $\{\gamma_t\}_{t\in[0,1]}$ be a smooth family of Riemannian metrics on $\Sigma$. Assume that $\gamma'_t<0$ for all $t\in[0,1]$. Fix $K>0$. Then there
exists a smooth Riemannian metric $g$ on $\Sigma\times[0, 1]$ with the following properties: 
\begin{itemize}[left=0.5cm]
\item [(1)] the slice metrics are prescribed by $\{\gamma_t\}_{t\in [0,1]}$, namely
$$g|_{\Sigma\times\{t\}}=\gamma_t;$$
\item [(2)] the Ricci curvature satisfies $\Ric^g\geq Kg;$
\item [(3)] the second fundamental form of $\Sigma\times\{0\}$ with respect to the inward unit normal
satisfies $\secf \geq K\gamma_0.$
\end{itemize}
\end{lemma}
\begin{remark}
This lemma strengthens the positive scalar curvature collar lemma in \cite{SWW2022}. In proving the positive scalar curvature extension theorem in \cite{SWW2022} and our positive Ricci curvature extension theorem (Theorem~\ref{Thm:main 1}), it suffices to prescribe only the endpoint metrics $\gamma_0$ and $\gamma_1$ with $\gamma_1<\gamma_0$. In that case, one may simply take the linear interpolation $\gamma_t=(1-t)\gamma_0+t\gamma_1$.
\end{remark}

\begin{proof}[Proof of Lemma \ref{Lem:PRicciCcobordismlm}]
We construct the desired metric $g$ in the quasi-spherical form \eqref{qsform} by choosing a suitable lapse function $u$.  Throughout this proof, $|\cdot|$ denotes the norm with respect to $g$; on tensors tangent to $\Sigma_t$, this agrees with the norm induced by $\gamma_t$.

We first compute the components of $\Ric^g$. Taking the $\gamma_t$-trace of \eqref{tn}, we obtain 
\begin{equation}\label{Riccinn}
\Ric^g_{nn}=\frac{u'}{2u^3}\tr_{\gamma_t}\gamma'_t+\frac{1}{4u^2}\tr_{\gamma_t}\left(\gamma'_t\right)^2-\frac{1}{2u^2}\tr_{\gamma_t} \gamma''_t.
\end{equation}
Contracting \eqref{tt} in the first and fourth indices with respect to $\gamma_t$ and then adding \eqref{tn}, we get 
\begin{equation}\label{Riccitt}
\begin{split}
\Ric^g_{ij}=\Ric^{\gamma_t}_{ij}+\frac{u'}{2u^3}\left(\gamma'_t\right)_{ij}+\frac{1}{2u^2}\left(\gamma'_t\right)^2_{ij}
-\frac{\tr_{\gamma_t}\gamma'_t}{4u^2}\left(\gamma'_t\right)_{ij}-\frac{1}{2u^2}\left(\gamma''_t\right)_{ij}.
\end{split}
\end{equation}
Contracting \eqref{mixed} in the second and third indices with respect to $\gamma_t$ yields
\begin{equation}\label{Riccimixed}
\Ric^g_{in}=\frac{1}{2u}\left(\Div_{\gamma_t}\gamma'_t-d\tr_{\gamma_t}\gamma'_t\right)_i.
\end{equation}

We now estimate $\Ric^g$. There exists $\Lambda>0$ such that, for all $t\in [0,1]$,
\begin{equation}\label{Ricciextensionupbound}
 \gamma'_t\geq-\Lambda\gamma_t,\quad \gamma''_t\leq\Lambda\gamma_t, \quad \Ric^{\gamma_t}\geq -\Lambda\gamma_t,\quad \big|\nabla^{\gamma_t}\gamma'_t\big|\leq\Lambda.
\end{equation}
Since $\gamma'_t<0$ for all $t\in [0,1]$, there exists $\lambda>0$ such that
\begin{equation}\label{Ricciextensionlowbound}
\gamma'_t\leq-\lambda\gamma_t.
\end{equation}
To match the negativity of $\gamma'_t$, we choose $u$ so that $u'<0$. 

Let $v$ be an arbitrary vector at a point of $\Sigma_t$. Decompose $v$ as 
$$v=v_T+se_n,$$
where $v_T$ is tangent to $\Sigma_t$ and $s\in\mathbb R$. Then
\begin{equation}\label{Riccigeneral}
\Ric^g(v,v)=s^2\Ric^g_{nn}+2s\Ric^g(v_T,e_n)+\Ric^g(v_T,v_T).
\end{equation}
By \eqref{Riccinn}, \eqref{Ricciextensionupbound}, and \eqref{Ricciextensionlowbound}, we have
\begin{equation}\label{Riccinnbound}
\Ric^g_{nn}\geq-\frac{\lambda u'}{2u^3}-\frac{C(n)\Lambda}{u^2},
\end{equation}
where $C(n)$ depends only on $n$ and may change from line to line. By \eqref{Riccitt}, \eqref{Ricciextensionupbound}, and \eqref{Ricciextensionlowbound}, we have
\begin{equation}\label{Riccittbound}
\Ric^g(v_T,v_T)\geq \left(-\frac{\lambda u'}{2u^3}-C(n)\frac{\Lambda^2+\Lambda}{u^2}-\Lambda\right)|v_T|^2.
\end{equation}
By \eqref{Riccimixed} and \eqref{Ricciextensionupbound}, we have
\begin{equation}\label{Riccimixedbound}
\left|\Ric^g(v_T,e_n)\right|\leq\frac{C(n)\Lambda }{u}|v_T|.
\end{equation}
Substituting \eqref{Riccinnbound}--\eqref{Riccimixedbound} into \eqref{Riccigeneral} yields
\begin{equation*}
\begin{split}
\Ric^g(v,v)\geq&\,\left(-\frac{\lambda u'}{2u^3}-\frac{C(n)\Lambda}{u^2}\right)s^2-\frac{C(n)\Lambda }{u}|s||v_T|\\
&+ \left(-\frac{\lambda u'}{2u^3}-C(n)\frac{\Lambda^2+\Lambda}{u^2}-\Lambda\right)|v_T|^2.
\end{split}
\end{equation*}
Applying Young’s inequality to the mixed term
$$\frac{\Lambda}{u}|s||v_T|\leq\frac{\Lambda s^2}{u^2}+\Lambda |v_T|^2,$$
we obtain
\begin{equation*}
\Ric^g(v,v)\geq\left(-\frac{\lambda u'}{2u^3}-\frac{C(n)\Lambda}{u^2}\right)s^2+\left(-\frac{\lambda u'}{2u^3}-C(n)\frac{\Lambda^2+\Lambda}{u^2}-C(n)\Lambda\right)|v_T|^2.
\end{equation*}
Since $|v|^2=|v_T|^2+s^2$, there exists a positive constant $C$ depending only on $\lambda$, $\Lambda$, $K$, and $n$ such that if $u$ satisfies
\begin{equation}\label{Rricciode}
u'\leq-C(u^3+u)
\end{equation}
on $[0,1]$, then $\Ric^g\geq Kg$ on $\Sigma\times[0,1]$. 

Finally, consider the second fundamental form of $\Sigma_0$. Since $e_n$ is the inward unit normal along $\Sigma_0$, we have $\secf=\secf_0$. Therefore, by \eqref{secf} and \eqref{Ricciextensionlowbound},
\[
\secf=-\frac{1}{2u(0)}\gamma'_0\geq\frac{\lambda}{2u(0)}\gamma_0.
\]
Hence, if 
\begin{equation}\label{Ricodeninitial value}
u(0)\leq\frac{\lambda}{2K},
\end{equation}
then property $(3)$ holds.

Take $u(t)=\varepsilon e^{-2Ct}$ with $\varepsilon=\min\{\lambda (2K)^{-1},1\}$. Then $u(0)\leq\lambda (2K)^{-1}$, which is precisely \eqref{Ricodeninitial value}. Moreover, $u(t)\le 1$ for $t\in[0,1]$, and thus
\[
u'=-2Cu\le-C(u^3+u),
\]
which verifies \eqref{Rricciode}. This completes the proof.
\end{proof}

\subsubsection{Necks with negative curvature operator} 
\begin{lemma}\label{Lem: NSecCcobordismlm}
Let $\Sigma$ be a closed $(n-1)$-dimensional manifold, and let $\{\gamma_t\}_{t\in[0,1]}$ be a smooth family of Riemannian metrics on $\Sigma$. Assume that $\gamma'_t>0$ for all $t\in[0,1]$. Fix $K>0$. Then there
exists a smooth Riemannian metric $g$ on $\Sigma\times[0, 1]$ with the following properties: 
\begin{itemize}[left=0.5cm]
\item [(1)] the slice metrics are prescribed by $\{\gamma_t\}_{t\in [0,1]}$, namely
$$g|_{\Sigma\times\{t\}}=\gamma_t;$$
\item [(2)] the curvature operator satisfies $\mathcal R^g\leq -K;$
\item [(3)] the second fundamental form of $\Sigma\times\{0\}$ with respect to the inward unit normal
satisfies $\secf\leq -K\gamma_0.$
\end{itemize}
\end{lemma}
This lemma is the sign-reversed analogue of the positive Ricci curvature collar lemma. However, the curvature-operator setting is more delicate. In the Ricci case, the mixed sectional curvatures of planes containing $\partial_t$ can be made sufficiently positive when $-u'$ is sufficiently large relative to $u+u^3$, so that they dominate the remaining contributions and force $\Ric>0$. For the curvature operator, one must also control the sectional curvatures of planes tangent to the slices $\Sigma_t$, and here we exploit the negative contribution in the Gauss equation when the second fundamental form of each $\Sigma_t$ has a definite sign.

\begin{proof}
We construct the desired metric $g$ in the quasi-spherical form
$$
g=u^2dt^2+\gamma_t
$$
by choosing an appropriate lapse function $u$. There exists a constant $\Lambda>0$ such that, for all $t\in [0,1]$,
\begin{equation}\label{secextensionupbound}
 \gamma'_t\leq \Lambda\gamma_t,\quad \gamma''_t\geq-\Lambda\gamma_t, \quad \mathcal R^{\gamma_t}\leq\Lambda,\quad\big|\nabla^{\gamma_t}\gamma'_t\big|\leq\Lambda.
\end{equation}
Since $\gamma'_t>0$ for all $t\in [0,1]$, there exists $\lambda>0$ such that
\begin{equation}\label{secextensionlowbound}
\gamma'_t\geq\lambda\gamma_t.
\end{equation}
To match the positivity of $\gamma'_t$, we choose $u$ so that $u'<0$. 

By \eqref{secextensionlowbound}, for any bivector $\varphi$ tangent to the slice $\Sigma_t$, we have
\begin{equation}\label{bivectorieq}
\left(\gamma'_t\wedge \gamma'_t\right)(\varphi,\varphi)\geq\lambda^2|\varphi|^2.
\end{equation}
Let $\omega$ be an arbitrary bivector at a point of $\Sigma_t$. Decompose $\omega$ as $$\omega=\psi+\beta\wedge e_n,$$ where
$\psi$ is a bivector and $\beta$ is a vector, both tangent to $\Sigma_t$. Then
\begin{equation}\label{curopeexpansion1}
\Rm^{g}\left(\omega, \omega\right)=\Rm^{g}\left(\psi,\psi\right)+2\Rm^{g}\left(\psi,\beta\wedge e_n\right)+\Rm^g\left(\beta\wedge e_n,\beta\wedge e_n\right).
\end{equation}
Using \eqref{tt}, \eqref{secextensionupbound}, and \eqref{bivectorieq}, we derive
\begin{equation}\label{estimate1}
\begin{split}
\Rm^g\left(\psi,\psi\right)&=\Rm^{\gamma_t}\left(\psi,\psi\right)-\frac{1}{4u^2}\left(\gamma'_t\wedge \gamma'_t\right)(\psi,\psi)\\
&\leq\left(\Lambda-\frac{\lambda^2}{4u^2}\right)|\psi|^2.
\end{split}
\end{equation}
Applying \eqref{mixed} and \eqref{secextensionupbound}, we obtain 
\begin{equation}\label{estimate2}
\Rm^{g}(\psi,\beta\wedge e_n)\leq \frac{\Lambda}{u}|\psi||\beta|,
\end{equation}
By \eqref{tn}, \eqref{secextensionupbound}, and \eqref{secextensionlowbound}, it follows that 
\begin{equation}\label{estimate3}
\Rm^g\left(\beta\wedge e_n,\beta\wedge e_n\right)\leq\left(\frac{\lambda u'}{2u^3}+\frac{\Lambda^2+\Lambda}{2u^2}\right)|\beta|^2.
\end{equation}
Substituting \eqref{estimate1}--\eqref{estimate3} into \eqref{curopeexpansion1} yields
\begin{equation*}
\Rm^{g}\left(\omega, \omega\right)\leq \left(\Lambda-\frac{\lambda^2}{4u^2}\right)|\psi|^2+\frac{2\Lambda}{u}|\psi||\beta|+\left(\frac{\lambda u'}{2u^3}+\frac{\Lambda^2+\Lambda}{2u^2}\right)|\beta|^2.
\end{equation*}
Applying Young’s inequality to the mixed term
$$\frac{2\Lambda}{u}|\psi||\beta|\leq\Lambda|\psi|^2+\frac{\Lambda}{u^2}|\beta|^2,$$
we get
\begin{equation*}
\Rm^{g}\left(\omega, \omega\right)\leq \left(2\Lambda-\frac{\lambda^2}{4u^2}\right)|\psi|^2+\left(\frac{\lambda u'}{2u^3}+\frac{\Lambda^2+3\Lambda}{2u^2}\right)|\beta|^2.
\end{equation*}
Since $|\omega|^2=|\psi|^2+|\beta|^2$, there exist constants $C_1,C_2>0$, depending only on $\lambda$, $\Lambda$, $K$, and $n$, such that if $u$ satisfies the following conditions on $[0,1]$:
$${\rm (i)}\ u\leq C_1,\qquad\  {\rm (ii)}\ u'\leq -C_2u,$$
then $\mathcal R^g\leq-K$ on $\Sigma\times [0,1]$. 

It remains to check the second fundamental form of $\Sigma_0$. As in the proof of Lemma~\ref{Lem:PRicciCcobordismlm}, we have
\[
\secf=-\frac{1}{2u(0)}\gamma'_0.
\]
Since $\gamma'_0\ge\lambda\gamma_0$, it follows that
\[
\secf\le -\frac{\lambda}{2u(0)}\gamma_0.
\]
Thus, property $(3)$ holds once $C_1$ in (i) is chosen sufficiently small. 

Finally, set $u=C_1 e^{-C_2t}$. Then $u$ satisfies (i) and (ii), and the proof is complete.
\end{proof}

\section{Corner smoothing}\label{Sec: 3}

Gluing two Riemannian manifolds with boundary along isometric boundary components is a common and useful operation. The resulting metric is generally only Lipschitz across the gluing hypersurface; this singularity is commonly referred to as a ``corner''. Under suitable assumptions on the two manifolds, the glued manifold may satisfy certain curvature bounds in a generalized sense (for example, in the Alexandrov sense). In some situations, one then needs to further smooth the corner while preserving these curvature bounds.

In this section, we first review Perelman’s corner-smoothing theorem for positive Ricci curvature, and then establish corresponding results under upper curvature bounds for sectional curvature and for the curvature operator. We begin with the following definition of a ``corner”:

\begin{definition}
A triple $(M, g, \Gamma)$ is called a {\it Riemannian manifold with a corner along $\Gamma$} if the following conditions are satisfied:
\begin{itemize}[left=0.5cm]
    \item [(i)] $M$ is a smooth manifold (possibly with boundary), and $\Gamma$ is a smoothly embedded, closed, two-sided hypersurface contained in the interior of $M$.
    \item [(ii)] $g$ is a Lipschitz Riemannian metric on $M$, smooth on $M \setminus \Gamma$, and smooth up to $\Gamma$ from each side.
\end{itemize}
\end{definition}

Since $\Gamma$ is two-sided, it admits a tubular neighborhood
\(
\mathcal{U} \cong (-\varepsilon,\varepsilon)\times\Gamma.
\)
Denote the two components of $\mathcal{U}\setminus\Gamma$ by
$\mathcal{U}_+$ and $\mathcal{U}_-$. Let $\secf^\pm$ denote the second
fundamental forms of $\Gamma$ from $\mathcal{U}_\pm$,
with respect to the unit normals pointing into $\mathcal{U}_\pm$.

If the metric has corners along finitely many pairwise disjoint hypersurfaces, one may choose disjoint tubular neighborhoods and perform the smoothing near each hypersurface independently. It therefore suffices to treat a single corner.

\subsection{Corner smoothing under lower curvature bounds}
For positive Ricci curvature, Perelman proved the following corner-smoothing theorem, which plays a key role in our proof of Theorem \ref{Thm:main 1}.
\begin{theorem}[Perelman \cite{Per1997}]\label{Lem: corner smoothing Ricci}
Let $(M,g,\Gamma)$ be a Riemannian manifold with a corner along $\Gamma$. Assume that
\begin{itemize}[left=0.5cm]
\item  [{\rm(i)}] $(M\setminus\Gamma,g)$ has positive Ricci curvature.
\item  [{\rm(ii)}] The sum $\secf^{+}+\secf^{-}$ is positive definite.
\end{itemize}
Then, for each open neighborhood $U$ of $\Gamma$, there exists a smooth metric $\tilde g$ on $M$ with positive Ricci curvature such that $\tilde g$ coincides with $g$ on $M\setminus U$.
\end{theorem}

Perelman provided a brief sketch, and several detailed proofs have since appeared; see, for example, \cite{BWW}. For any $K\in\mathbb R$, the smoothing can likewise be carried out while preserving $\Ric>K$; see \cite[Corollary B]{RW232}. Moreover, analogous corner-smoothing results preserving lower bounds for other types of curvature have been established. For sectional curvature and curvature operator, see \cite{Kos1} and \cite{Chow,Sch12}, respectively. In the scalar-curvature setting, the condition that $\secf^{+}+\secf^{-}$ be positive can be weakened to require that $\tr_{\Gamma}(\secf^{+}+\secf^{-})$ (i.e., the sum of the mean curvatures) be positive; see \cite{BH,BMN,Miao2002,Sch12}. For $k^{\rm th}$-intermediate Ricci curvature and intermediate curvatures between Ricci and scalar curvature, see \cite{RW232}.

\subsection{Corner smoothing under upper curvature bounds}

In contrast to the case of lower curvature bounds, corner smoothing under upper curvature bounds is more subtle: merely reversing the sign condition on the sum of the second fundamental forms is insufficient to preserve the upper curvature bound. In fact, Kosovskiĭ showed that even preserving the upper curvature bound in the Alexandrov sense (without smoothing the glued metric) requires additional conditions. Here, having curvature bounded above by \(K\) in the Alexandrov sense means that, with respect to the induced length metric, sufficiently small geodesic triangles are no thicker than comparison triangles with the same side lengths in the model surface of constant curvature \(K\). This is also known as the locally \(\mathrm{CAT}(K)\) condition, whereas the \(\mathrm{CAT}(K)\) condition is its global counterpart.

\begin{theorem}[Kosovskiĭ \cite{Kos2}]\label{Kos2}
Let \((M,g,\Gamma)\) be a Riemannian manifold with a corner along \(\Gamma\), and assume that \(M\) has no boundary. Fix \(K\in\mathbb R\). Then \((M,g)\) is locally \(\mathrm{CAT}(K)\) if and only if the following conditions hold:
\begin{itemize}[left=0.5cm]
\item  [{\rm(i)}] $(M\setminus\Gamma,g)$ has sectional curvature no greater than $K$.
\item  [{\rm(ii)}] The sum $\secf^{+}+\secf^{-}$ is negative semidefinite.
\item [{\rm(iii)}] For every $p\in \Gamma$ and every two-plane $\Pi\subset T_p\Gamma$ such that
$\secf^{+}|_{\Pi}$ and $\secf^{-}|_{\Pi}$ are both negative definite, it holds that
$\sec^{\,g|_{\Gamma}}(\Pi)\le K$.
\end{itemize}
\end{theorem}

In Kosovski\u{\i}'s proof, smoothing provides a preliminary upper curvature bound, while the precise Alexandrov upper curvature bound $K$ is obtained by establishing the $K$-convexity of normal Jacobi fields, following \cite{ABB}. More generally, there is a distinction between synthetic and smooth nonpositive curvature: Davis, Januszkiewicz, and Lafont \cite{DJL} constructed closed smooth $4$-manifolds admitting locally $\mathrm{CAT}(0)$ metrics but no smooth Riemannian metrics of nonpositive sectional curvature. For the smoothing procedure, we impose a $2$-convexity condition that is stronger than condition (iii) and suffices for our application to Theorem~\ref{Thm:main 2}. This smoothing procedure also applies to the curvature operator.

\begin{theorem}\label{thm: smoothing sec}
Let $(M,g,\Gamma)$ be a Riemannian manifold with a corner along $\Gamma$. Assume:
\begin{itemize}[left=0.5cm]
\item [{\rm(i)}] $(M\setminus\Gamma,g)$ has sectional curvature (respectively, curvature operator) bounded above by $K$ for some constant $K$.
\item [\rm(ii)] The sum $\secf^{+}+\secf^{-}$ is negative definite.
\item [\rm(iii)] Up to interchanging the two sides, $\secf^-$ is $2$-convex with respect to $-(\secf^++\secf^-)$. That is, for every $p\in\Gamma$ and every orthonormal pair $\{e_1,e_2\}\subset T_p\Gamma$ with respect to $-(\secf^++\secf^-)$, one has $\secf^-(e_1,e_1)+\secf^-(e_2,e_2)\geq 0$.
\end{itemize}
Then, for any open neighborhood $U$ of $\Gamma$ and any $\varepsilon>0$, there exists a smooth metric $\tilde g$ on $M$ whose sectional curvature (respectively, curvature operator) is bounded above by $K+\varepsilon$, with $\tilde g=g$ on $M\setminus U$ and $\|\tilde g-g\|_{C^0}<\varepsilon$.
\end{theorem}

In the proof of Theorem \ref{Thm:main 2}, we will verify assumption (iii) by checking the following easier-to-apply sufficient condition.

\begin{proposition}\label{practical_use}
Assume condition (ii) of Theorem \ref{thm: smoothing sec}. Suppose that, after possibly interchanging the two sides, $\secf^+<0$ and $\secf^-$ is $2$-convex with respect to $-\secf^+$. Then condition (iii) of Theorem \ref{thm: smoothing sec} holds.
\end{proposition}

The proofs of Theorem \ref{thm: smoothing sec} in the sectional-curvature case and in the curvature-operator case follow the same overall strategy, but the latter is slightly more involved. We therefore present the proof in the curvature-operator case and then discuss the modifications needed for the sectional-curvature case. Throughout the proofs, we use the notation and conventions introduced in Subsection \ref{subsubsec: notation and conventions}.

\begin{proof}[Proof of Theorem \ref{thm: smoothing sec} (curvature-operator case)]
Choose $\delta_0>0$ sufficiently small so that the normal exponential map identifies $\Gamma\times[-\delta_0,\delta_0]$ with a tubular neighborhood $\mathcal U$ of $\Gamma$ contained in $U$. Let $t$ be the signed distance coordinate, with the sign
chosen so that $\mathcal{U}_-$ corresponds to $\{t<0\}$. In these coordinates, we can write $g$ as
\[
g = dt^2 + \gamma(t),
\]
where $\gamma(t)$ is a family of smooth metrics on $\Gamma$, smooth in $t$ on $[-\delta_0,0]$ and on $[0,\delta_0]$, and Lipschitz across $t=0$. Under our sign convention for the second
fundamental form, the one-sided derivatives at $t=0$ satisfy
\[
\gamma'(0^-) = 2\secf^- \quad\text{and}\quad \gamma'(0^+) = -2\secf^+.
\]
The proof proceeds in two steps: a $C^1$-interpolation and a standard smoothing. 

{\it Step 1. $C^1$-interpolation.} We adopt Perelman’s polynomial interpolation scheme \cite{Per1997}. The idea is to replace $g$ on a small neighborhood of $\Gamma$ by a cubic interpolation, thereby obtaining a globally $C^1$ metric. We then show that this interpolation almost preserves the required curvature-operator upper bound. 

For $0<\delta<\delta_0$, let $\gamma_\delta(t)$ be the cubic Hermite interpolant on $[-\delta,\delta]$ determined by the endpoint data $\gamma(\pm\delta)$ and $\gamma'(\pm\delta)$, namely, 
\[
\gamma_\delta(\pm\delta)=\gamma(\pm\delta),\qquad \gamma'_\delta(\pm\delta)=\gamma'(\pm\delta).
\]
An explicit formula is
 \[
\begin{split}
\gamma_{\delta}(t)=\,&\frac{t+\delta}{2\delta}\gamma(\delta)-\frac{t-\delta}{2\delta}\gamma(-\delta)+\frac{(t-\delta)^2(t+\delta)}{4\delta^2}\left[\gamma'(-\delta)-\frac{1}{2\delta}\left(\gamma(\delta)-\gamma(-\delta)\right)\right]\\
&+\frac{(t+\delta)^2(t-\delta)}{4\delta^2}\left[\gamma'(\delta)-\frac{1}{2\delta}\left(\gamma(\delta)-\gamma(-\delta)\right)\right].
\end{split}
\]
On $\Gamma\times[-\delta,\delta]$, define \[g_\delta := dt^2+\gamma_\delta(t),\] and set $g_\delta:=g$ on $M\setminus(\Gamma\times[-\delta,\delta])$. Write $\Gamma_t:=\Gamma\times\{t\}$.
Then $g_\delta$ is $C^{1}$ across the two hypersurfaces $\Gamma_{\pm\delta}$ and smooth elsewhere.

To estimate the curvature of $g_\delta$, we first control $\gamma_\delta(t)$, $\gamma'_\delta(t)$, and $\gamma''_\delta(t)$ on $[-\delta,\delta]$. Since the family $\gamma(t)$ is smooth on $[-\delta,0]$ and on $[0,\delta]$, we have
\begin{equation}\label{continuity}
\|\gamma(\pm\delta)-\gamma(0)\|_{C^2(\Gamma)}=O(\delta), \qquad\|\gamma'(\pm\delta)-\gamma'(0^\pm)\|_{C^2(\Gamma)}=O(\delta),
\end{equation}
and
$$
\left\|\frac{1}{\delta}\left(\gamma(\delta)-\gamma(-\delta)\right)-\left(\gamma'(0^+)+\gamma'(0^-)\right)\right\|_{C^2(\Gamma)}=O(\delta).
$$
Here and throughout, the $C^k$-norms on $\Gamma$ are taken with respect to $\gamma(0)$ (we suppress this in the notation). We also write $O(\delta)$ for a quantity bounded by $C\delta$, and $O(1)$ for a quantity bounded by $C$, where $C$ is independent of $\delta$ and may change from line to line.

It follows that
\begin{equation}\label{Eq: C0}
\sup_{t\in[-\delta,\delta]}\|\gamma_{\delta}(t)-\gamma(0)\|_{C^2(\Gamma)}=O(\delta),
\end{equation}
\begin{equation}\label{Eq: C1}
\sup_{t\in[-\delta,\delta]}\left\|\gamma'_{\delta}(t)-\left(-\frac{\delta+t}{\delta}\secf^{+}+\frac{\delta-t}{\delta}\secf^{-}\right)\right\|_{C^2(\Gamma)}=O(\delta),
\end{equation}
and
\begin{equation}\label{Eq: C2}
\sup_{t\in[-\delta,\delta]}\left\|\gamma''_{\delta}(t)+\frac{1}{\delta}\left(\secf^{+}+\secf^{-}\right)\right\|_{C^2(\Gamma)}=O(1).
\end{equation}
For the derivation of \eqref{Eq: C0}--\eqref{Eq: C2}, see \cite[Lemmas 3--4]{BWW}, \cite[Appendix A.2.1.1]{Burdick},  or \cite[Lemma 2.1]{RW232}. With these preparations, we turn to the curvature estimates of $g_\delta$ on $\Gamma\times[-\delta,\delta]$.

{\it Substep 1.1. Tangential-normal curvatures.} Applying \eqref{tn} to $g_\delta$ with $u\equiv 1$, we obtain
\begin{equation}\label{C1tnrelation}
R^{g_\delta}_{nijn}=-\frac{1}{2}\left(\gamma''_\delta(t)\right)_{ij}+\frac{1}{4}\left(\gamma'_\delta(t)\right)^2_{ij}.
\end{equation}
By \eqref{Eq: C0} and \eqref{Eq: C1}, there exists a positive constant $\Lambda_1$ independent of $\delta$ such that, for all $t\in[-\delta,\delta]$,
\begin{equation}\label{quadraricuobound}
\left(\gamma'_\delta(t)\right)^2\leq\Lambda_1\gamma_\delta(t).
\end{equation}
Since $\secf^{+}+\secf^{-}$ is negative definite, there exists $\lambda>0$ such that
\begin{equation}\label{sumofsec}
-(\secf^{+}+\secf^{-})\geq\lambda\gamma(0).
\end{equation}
By \eqref{Eq: C0}, \eqref{Eq: C2}, and \eqref{sumofsec}, for sufficiently small $\delta$ and all $t\in[-\delta,\delta]$, we have
\begin{equation}\label{secondderlowbound}
\gamma''_\delta(t)\geq\frac{\lambda}{2\delta}\gamma_\delta(t).
\end{equation}
Substituting \eqref{quadraricuobound} and \eqref{secondderlowbound} into \eqref{C1tnrelation}, we conclude that $R^{g_\delta}(e_n,\cdot,\cdot,e_n)$, viewed as a symmetric $2$-tensor field on each slice $\Gamma_t$, satisfies
\begin{equation}\label{smoothupboundtn}
R^{g_\delta}(e_n,\cdot,\cdot,e_n)\leq -\frac{\lambda}{4\delta}\gamma_\delta(t)+\Lambda_1\gamma_\delta(t)\leq -\frac{\lambda}{8\delta}\gamma_\delta(t)
\end{equation}
for all $t\in[-\delta,\delta]$, provided $\delta$ is sufficiently small.

{\it Substep 1.2. Tangential curvatures.} Let $\varphi$ be an arbitrary bivector tangent to $\Gamma_t$. Applying \eqref{tt} to $g_\delta$ with $u\equiv 1$ on $\Gamma_t$, we obtain
\begin{equation}\label{curvoprgauss}
\Rm^{g_\delta}(\varphi,\varphi)=\Rm^{\gamma_\delta(t)}(\varphi,\varphi)-\frac{1}{4}\left(\gamma_{\delta}'(t)\wedge\gamma_{\delta}'(t)\right)(\varphi,\varphi).
\end{equation}
Via the local product structure on the tubular neighborhood, we identify each slice $\Gamma_s$ with $\Gamma$, and we use the same symbol $\varphi$ for the corresponding bivector on each slice. Applying \eqref{tt} to $g$ with $u\equiv 1$ on $\Gamma_{-\delta}$, we get
\begin{equation*}
\Rm^{g}(\varphi,\varphi)=\Rm^{\gamma(-\delta)}(\varphi,\varphi)-\frac{1}{4}\left(\gamma'(-\delta)\wedge\gamma'(-\delta)\right)(\varphi,\varphi).
\end{equation*}
Since $\mathcal R^g\leq K$ by assumption, we have
\begin{equation}\label{assucuroprupbound}
\Rm^{\gamma(-\delta)}(\varphi,\varphi)\leq K|\varphi|^2_{\gamma(-\delta)}+\frac{1}{4}\left(\gamma'(-\delta)\wedge\gamma'(-\delta)\right)(\varphi,\varphi).
\end{equation}
By \eqref{continuity} and \eqref{Eq: C0}, the metrics $\gamma_\delta(t)$ and $\gamma(-\delta)$ are $C^2$-close. Consequently,
\begin{equation}\label{curopercontinuity}
\Rm^{\gamma_\delta(t)}(\varphi,\varphi)=\Rm^{\gamma(-\delta)}(\varphi,\varphi)+O(\delta)\,|\varphi|^2_{\gamma(-\delta)}.
\end{equation}
We use the notation $O(\delta)A$ to make the dependence on the factor $A$ explicit: $O(\delta)A$ denotes a quantity $E$ with $|E|\le C\delta A$, where $C$ is a constant independent of $\delta$ and $A$, uniformly for $t\in [-\delta,\delta]$, and $C$ may change from line to line.

Combining \eqref{curvoprgauss}, \eqref{assucuroprupbound}, and \eqref{curopercontinuity}, we obtain
\begin{equation}\label{curvoprgauss'} 
\begin{split}
\Rm^{g_\delta}(\varphi,\varphi)\leq&\left(K+O(\delta)\right)|\varphi|^2_{\gamma(-\delta)}+\frac{1}{4}\left(\gamma'(-\delta)\wedge\gamma'(-\delta)\right)(\varphi,\varphi)\\
&-\frac{1}{4}\left(\gamma_{\delta}'(t)\wedge\gamma_{\delta}'(t)\right)(\varphi,\varphi).
\end{split}
\end{equation}
By \eqref{continuity}, we have 
\begin{equation}\label{secfcontinuity2}
\left(\gamma'(-\delta)\wedge\gamma'(-\delta)\right)(\varphi,\varphi)=4(\secf^-\wedge\secf^-)(\varphi,\varphi)+O(\delta)\,|\varphi|^2_{\gamma(0)}.
\end{equation}
For convenience, let 
\[A_t=-\frac{\delta+t}{\delta}\secf^{+}+\frac{\delta-t}{\delta}\secf^{-}.\] By \eqref{Eq: C1}, we have
\begin{equation}\label{secfcontinuity1}
\left(\gamma_{\delta}'(t)\wedge\gamma_{\delta}'(t)\right)(\varphi,\varphi)=(A_t\wedge A_t)(\varphi,\varphi)+O(\delta)\,|\varphi|_{\gamma(0)}^2.
\end{equation}
Substituting \eqref{secfcontinuity2} and \eqref{secfcontinuity1} into \eqref{curvoprgauss'} and using the uniform equivalence of the norms induced by $\gamma(-\delta)$, $\gamma(0)$, and $\gamma_\delta(t)$, we arrive at
\begin{equation}\label{smoothupboundttpre}
\Rm^{g_\delta}(\varphi,\varphi)\leq \left(K+O(\delta)\right)|\varphi|^2_{\gamma_\delta(t)}+(\secf^-\wedge\secf^-)(\varphi,\varphi)-\frac{1}{4}(A_t\wedge A_t)(\varphi,\varphi).
\end{equation}

Next, we prove 
\begin{equation}\label{secfmono}
(A_t\wedge A_t)(\varphi,\varphi)\geq 4(\secf^-\wedge\secf^-)(\varphi,\varphi).
\end{equation}
Via the local product identification, we view $\varphi$ at its base point $(p,t)\in\Gamma_t$ as a bivector in $\wedge^2T_p\Gamma$. By the positivity assumption, $-\left(\secf^{+}+\secf^{-}\right)$ defines an inner product on $T_p\Gamma$. With respect to this inner product, $\secf^{-}$ determines a self-adjoint endomorphism of $T_p\Gamma$. Let $\{\lambda_1,\cdots,\lambda_{n-1}\}$ be its eigenvalues, and let $\{v_1,\cdots,v_{n-1}\}$ be the corresponding orthonormal eigenvectors. Thus, for $1\leq i,j\leq n-1$, we have
\begin{equation*}
-\left(\secf^{+}+\secf^{-}\right)(v_i,v_j)=\delta_{ij}, \qquad \secf^{-}(v_i,v_j)=\lambda_i\delta_{ij}.
\end{equation*}
The $2$-convexity assumption implies that $\lambda_i+\lambda_j\geq 0$ whenever $i\neq j$.

Write $\varphi$ as $\varphi= \sum_{1\leq i<j\leq n-1} \alpha^{ij} v_i \wedge v_j$. Denote $(A_t\wedge A_t)(\varphi,\varphi)$ by $F(t)$. Then we have 
\begin{equation*}
F(t)=\sum_{1\leq i<j\leq n-1}\left(\alpha^{ij}\right)^2\left(\frac{\delta+t}{\delta}+2\lambda_i\right)\left(\frac{\delta+t}{\delta}+2\lambda_j\right).
\end{equation*}
A direct calculation yields
\begin{equation*}
F''(t)\equiv\frac{2}{\delta^{2}}\sum_{1\leq i<j\leq n-1}\left(\alpha^{ij}\right)^2\geq 0,
\end{equation*}
and
\begin{equation*}
F'(-\delta)=\frac{2}{\delta}\sum_{1\leq i<j\leq n-1}(\lambda_i+\lambda_j)\left(\alpha^{ij}\right)^2.
\end{equation*}
By the $2$-convexity assumption, $F'(-\delta)\geq 0$. It follows that $F(t)\geq F(-\delta)$ for all $t\in[-\delta,\delta]$, i.e., \eqref{secfmono} holds. Substituting \eqref{secfmono} into \eqref{smoothupboundttpre}, we finally obtain 
\begin{equation}\label{smoothupboundtt}
\Rm^{g_\delta}(\varphi,\varphi)\leq \left(K+O(\delta)\right)|\varphi|^2_{\gamma_\delta(t)}.
\end{equation}

{\it Substep 1.3. Mixed curvature terms.} By \eqref{Eq: C0} and \eqref{Eq: C1}, there exists a constant $\Lambda_2>0$ independent of $\delta$, such that
\begin{equation}\label{mixed-bound-nabla}
\sup_{t\in[-\delta,\delta]}\Big\|\nabla^{\gamma_\delta(t)}\gamma'_\delta(t)\Big\|_{C^0(\Gamma)}\le\Lambda_2.
\end{equation}
Let $\varphi$ be an arbitrary bivector and $\xi$ an arbitrary vector, both tangent to $\Gamma_t$. Applying \eqref{mixed} to $g_\delta$ with $u\equiv 1$, and using \eqref{Eq: C0} and \eqref{mixed-bound-nabla}, we obtain
\begin{equation}\label{smoothupboundmixed}
\left|\Rm^{g_\delta}(\varphi,\xi\wedge e_n)\right|\leq 2\Lambda_2|\varphi|_{\gamma_\delta(t)}|\xi|_{\gamma_\delta(t)}.
\end{equation}

{\it Substep 1.4. General curvatures.} 
Let $\omega$ be an arbitrary bivector at a point of $\Gamma_t$. Decompose $\omega$ as 
\[\omega=\psi+\beta\wedge e_n,\] 
where
$\psi$ is a bivector and $\beta$ is a vector, both tangent to $\Gamma_t$. Then we have
\begin{equation}\label{curoprexpansion21}
\Rm^{g_\delta}\left(\omega,\omega\right)=\Rm^{g_\delta}\left(\psi,\psi\right)+2\Rm^{g_\delta}\left(\psi,\beta\wedge e_n\right)+\Rm^{g_\delta}\left(\beta\wedge e_n,\beta\wedge e_n\right).
\end{equation}
Applying \eqref{smoothupboundtn}, \eqref{smoothupboundtt}, and \eqref{smoothupboundmixed} to \eqref{curoprexpansion21} yields
\begin{equation*}
\Rm^{g_\delta}\left(\omega,\omega\right)\leq \left(K+O(\delta)\right)|\psi|^2_{\gamma_\delta(t)}+4\Lambda_2|\psi|_{\gamma_\delta(t)}|\beta|_{\gamma_\delta(t)}-\frac{\lambda}{8\delta}|\beta|^2_{\gamma_\delta(t)}.
\end{equation*}
Using Young's inequality for the mixed term, we get
\begin{equation*}
4\Lambda_2|\psi|_{\gamma_\delta(t)}|\beta|_{\gamma_\delta(t)}\leq\frac{64\Lambda_2^2\delta}{\lambda} |\psi|^2_{\gamma_\delta(t)}+\frac{\lambda}{16\delta}|\beta|^2_{\gamma_\delta(t)}.
\end{equation*}
Therefore,
\begin{equation*}
\Rm^{g_\delta}\left(\omega,\omega\right)\leq \left(K+O(\delta)\right)|\psi|^2_{\gamma_\delta(t)}-\frac{\lambda}{16\delta}|\beta|^2_{\gamma_\delta(t)}.
\end{equation*}
Noting that $|\omega|^2_{g_\delta}=|\psi|^2_{\gamma_\delta(t)}+|\beta|^2_{\gamma_\delta(t)}$, we arrive at
\begin{equation*}
\Rm^{g_\delta}\left(\omega,\omega\right)\leq \left(K+O(\delta)\right)|\omega|^2_{g_\delta}.
\end{equation*}
Thus, for $\delta$ sufficiently small, we obtain $\mathcal R^{g_\delta}\le K+\varepsilon/2$ on $\Gamma\times[-\delta,\delta]$. Since $g_\delta=g$ outside this region, the same bound holds globally. 

{\it Step 2. Standard mollification.}
 After Step 1, we obtain a Riemannian manifold with corners along two hypersurfaces, both contained in the neighborhood $U$. By construction, the metric is $C^{1}$ across each corner. Away from the corners, the metric is smooth and has curvature operator bounded above by $K+\varepsilon/2$.

Next we smooth the two $C^1$ corners. Since this process can be carried out independently at each corner, it suffices to treat one of them, say the one along $\Gamma_\delta$. For notational convenience, we shift the $t$-coordinate so that the corner lies at $t=0$. After shrinking $\delta_0$ if necessary, we still have a normal-coordinate neighborhood $\Gamma\times[-\delta_0,\delta_0]\subset U$. We relabel the resulting $C^{1}$ metric (formerly $g_\delta$) by $g$. On $\Gamma\times[-\delta_0,\delta_0]$, $g$ has the form
\[
g = dt^2 + \gamma(t),
\]
where $\{\gamma(t)\}_{t \in [-\delta_0, \delta_0]}$ is a family of smooth metrics on $\Gamma$, depending smoothly on $t$ for $t\in[-\delta_0,0]$ and for $t\in[0,\delta_0]$, and $C^{1}$ across $t=0$.

Instead of applying a further interpolation step to upgrade the metric from \(C^{1}\) to \(C^{2}\) in Perelman’s construction (as presented in \cite{BWW}), we follow the cutoff-and-mollify approach of Reiser--Wraith. In \cite[Lemma 3.1]{RW23} they establish a mollification lemma for \(C^{1}\) functions of one variable whose second derivative has a jump discontinuity. This lemma is applied in \cite{RW232} to smooth a \(C^{1}\) metric across a corner by mollifying its local coefficient functions. For the reader’s convenience, we give a self-contained account of the argument; for our purposes, it is more convenient to formulate it directly in terms of the slice metrics \(\gamma(t)\).

Let $\eta:(-\infty,+\infty)\to [0,+\infty)$ be a nonnegative even bump function supported in $(-1,1)$, normalized so that $\int_{-\infty}^{+\infty}\eta(t)\,dt=1$. Also, let $\phi:(-\infty,+\infty)\to [0,1]$ be a cutoff function supported in $(-\delta_0,\delta_0)$ such that $\phi\equiv 1$ on $[-\delta_0/2,\delta_0/2]$. On $[-\delta_0,\delta_0]$, for $0<\delta<\delta_0/2$, define
\[
\gamma_{\delta}(t)=\phi(t)\int_{-1}^{1}\eta(s)\gamma\left(t-\delta s\right)ds+\left(1-\phi(t)\right)\gamma(t).
\]
A standard mollification estimate yields
\begin{equation}\label{secondsmoothestimate1}
\left\|\gamma_{\delta}(t)-\gamma(t)\right\|_{C^2(\Gamma)}=O(\delta).
\end{equation}
Differentiating with respect to $t$, we obtain
\[
\begin{split}
\gamma'_{\delta}(t)-\gamma'(t)=\ &\phi(t)\int_{-1}^{1}\eta(s)\left(\gamma'(t-\delta s)-\gamma'(t)\right)ds\\
&+\phi'(t)\int_{-1}^{1}\eta(s)\left(\gamma(t-\delta s)-\gamma(t)\right)ds,
\end{split}
\]
and hence
\begin{equation}\label{secondsmoothestimate2}
\left\|\gamma'_{\delta}(t)-\gamma'(t)\right\|_{C^2(\Gamma)}=O(\delta).
\end{equation}

For the second $t$-derivative, note that if $|t|\ge \delta$, then $t-\delta s$ stays on a single smooth side of the corner for all $s\in[-1,1]$. Differentiating once more, we obtain
\begin{equation*}
\begin{split}
\gamma''_{\delta}(t)-\gamma''(t)=\ &\phi(t)\int_{-1}^{1}\eta(s)\left(\gamma''(t-\delta s)-\gamma''(t)\right)ds\\
&+2\phi'(t)\int_{-1}^{1}\eta(s)\left(\gamma'(t-\delta s)-\gamma'(t)\right)ds\\
&+\phi''(t)\int_{-1}^{1}\eta(s)\left(\gamma(t-\delta s)-\gamma(t)\right)ds.
\end{split}
\end{equation*}
It follows that for $|t|\ge \delta$,
\begin{equation}\label{secondsmoothestimate3}
\left\|\gamma''_{\delta}(t)-\gamma''(t)\right\|_{C^2(\Gamma)}=O(\delta).
\end{equation}
On $(-\delta,\delta)$, we have $\phi(t)\equiv 1$, and hence
\[
\gamma_{\delta}(t)=\int_{-1}^{1}\eta(s)\gamma(t-\delta s)\,ds.
\]
Consequently, for $t\in(-\delta,\delta)$, $\gamma''_{\delta}(t)$ satisfies the following approximate interpolation property:
\begin{equation}\label{secondsmoothestimate4}
\left\|\gamma''_{\delta}(t)-\Big(\theta_t\gamma''(0^+)+(1-\theta_t)\gamma''(0^-)\Big)\right\|_{C^2(\Gamma)}=O(\delta),
\end{equation}
where
\[
\theta_t=\int_{-1}^{t/\delta}\eta(s)\,ds.
\]
On $\Gamma\times [-\delta_0,\delta_0]$, we replace $g$ by 
\[g_\delta=dt^2+\gamma_{\delta}(t).\] The resulting metric is smooth everywhere; it remains to verify the curvature requirements. 

Let $\omega$ be an arbitrary bivector at a point of $\Gamma_t$. Write $\omega$ as $$\omega=\psi+\beta\wedge e_n,$$ where
$\psi$ is a bivector and $\beta$ a vector, both tangent to $\Gamma_t$. Then we have
\begin{equation*}
\Rm^{g_\delta}\left(\omega,\omega\right)=\Rm^{g_\delta}\left(\psi,\psi\right)+2\Rm^{g_\delta}\left(\psi,\beta\wedge e_n\right)+\Rm^{g_\delta}\left(\beta\wedge e_n,\beta\wedge e_n\right).
\end{equation*}

We first estimate $\Rm^{g_\delta}\left(\beta\wedge e_n,\beta\wedge e_n\right)$. Applying \eqref{tn} to $g_\delta$ with $u\equiv 1$ on $\Gamma_{t}$, we obtain
\begin{equation}\label{tn-gdelta}
\Rm^{g_\delta}\left(\beta\wedge e_n,\beta\wedge e_n\right)=-\frac{1}{2}\gamma''_\delta(t)(\beta,\beta)+\frac{1}{4}\left(\gamma'_\delta(t)\right)^2(\beta,\beta).
\end{equation}
Similarly, for $g$ and $t\neq 0$, we have
\begin{equation}\label{tn-g}
\Rm^{g}\left(\beta\wedge e_n,\beta\wedge e_n\right)=-\frac{1}{2}\gamma''(t)(\beta,\beta)+\frac{1}{4}\left(\gamma'(t)\right)^2(\beta,\beta).
\end{equation}
Letting $t\to 0^\pm$, we define $\Rm^{g_\pm}$ as the corresponding one-sided limits, so that
\begin{equation}\label{tn-gpm}
\Rm^{g_\pm}\left(\beta\wedge e_n,\beta\wedge e_n\right)=-\frac{1}{2}\gamma''(0^\pm)(\beta,\beta)+\frac{1}{4}\left(\gamma'(0)\right)^2(\beta,\beta).
\end{equation}
For $|t|\geq \delta$, comparing \eqref{tn-gdelta} with \eqref{tn-g} and using \eqref{secondsmoothestimate2}, \eqref{secondsmoothestimate3}, we deduce
\begin{equation}\label{tnsecsmoothcase1}
\Rm^{g_\delta}\left(\beta\wedge e_n,\beta\wedge e_n\right)=\Rm^{g}\left(\beta\wedge e_n,\beta\wedge e_n\right)+O(\delta)\,|\beta|^2_{\gamma(0)}.
\end{equation}
For $|t|<\delta$, comparing \eqref{tn-gdelta} with \eqref{tn-gpm} and using \eqref{secondsmoothestimate2}, \eqref{secondsmoothestimate4}, we arrive at
\begin{equation}\label{tnsecsmoothcase2}
\begin{split}
\Rm^{g_\delta}\left(\beta\wedge e_n,\beta\wedge e_n\right)
=&\ \theta_t\Rm^{g_+}\left(\beta\wedge e_n,\beta\wedge e_n\right)\\
&+\left(1-\theta_t\right)\Rm^{g_-}\left(\beta\wedge e_n,\beta\wedge e_n\right)+O(\delta)\,|\beta|^2_{\gamma(0)}.
\end{split}
\end{equation}

For $\Rm^{g_\delta}\left(\psi,\psi\right)$, by \eqref{tt}, \eqref{secondsmoothestimate1} and \eqref{secondsmoothestimate2}, we obtain
\begin{equation}\label{ttsecsmooth}
\begin{split}
\Rm^{g_\delta}(\psi,\psi)&=\Rm^{\gamma_\delta(t)}(\psi,\psi)-\frac{1}{4}\left(\gamma_{\delta}'(t)\wedge\gamma_{\delta}'(t)\right)(\psi,\psi)\\
&=\Rm^{\gamma(t)}(\psi,\psi)-\frac{1}{4}\left(\gamma'(t)\wedge\gamma'(t)\right)(\psi,\psi)+O(\delta)\,|\psi|_{\gamma(0)}^2\\
&=\Rm^{g}(\psi,\psi)+O(\delta)\,|\psi|_{\gamma(0)}^2.
\end{split}
\end{equation}

For the mixed term $\Rm^{g_\delta}(\psi,\beta\wedge e_n)$, note that it can be written in terms of $\nabla^{\gamma_\delta(t)}\gamma'_\delta(t)$ (cf. \eqref{mixed}); in particular, it involves no second $t$-derivatives of $\gamma_\delta(t)$. Arguing as above, we get
\begin{equation}\label{mixsecsmooth}
\Rm^{g_\delta}(\psi,\beta\wedge e_n)=\Rm^{g}(\psi,\beta\wedge e_n)+O(\delta)\,|\psi|_{\gamma(0)}|\beta|_{\gamma(0)}.
\end{equation}

We return to $\Rm^{g_\delta}\left(\omega,\omega\right)$. Note that $|\omega|^2_{g_\delta}=|\psi|^2_{\gamma_\delta(t)}+|\beta|^2_{\gamma_\delta(t)}$, and that the norms induced by $\gamma_\delta(t)$ and $\gamma(0)$ are uniformly equivalent. For $|t|\geq\delta$, \eqref{tnsecsmoothcase1}, \eqref{ttsecsmooth} and \eqref{mixsecsmooth} imply
\begin{equation}
\Rm^{g_\delta}\left(\omega,\omega\right)=\Rm^{g}\left(\omega,\omega\right)+O(\delta)\,|\omega|^2_{g_\delta}.
\end{equation}
For $|t|<\delta$, \eqref{tnsecsmoothcase2}, \eqref{ttsecsmooth} and \eqref{mixsecsmooth} yield
\begin{equation}
\Rm^{g_\delta}\left(\omega,\omega\right)=\theta_t\Rm^{g_+}\left(\omega,\omega\right)+(1-\theta_t)\Rm^{g_-}\left(\omega,\omega\right)+O(\delta)\,|\omega|^2_{g_\delta}.
\end{equation}
In either case, we have 
\[
\Rm^{g_\delta}\left(\omega,\omega\right)\leq\left(K+\frac{\varepsilon}{2}+O(\delta)\right)|\omega|^2_{g_\delta}.
\]
Thus, for $\delta$ sufficiently small, we conclude that $\mathcal R^{g_\delta}\le K+\varepsilon$ on $\Gamma\times[-\delta_0,\delta_0]$.

Finally, define $\tilde g$ on $M$ by
\begin{equation*}
\tilde g=\left\{
\begin{aligned}
&\,g\quad\ \mbox{on}\ M\setminus\left(\Gamma\times[-\delta_0,\delta_0]\right),\\
&\,g_\delta\quad\mbox{on}\ \Gamma\times[-\delta_0,\delta_0].
\end{aligned}
\right.
\end{equation*}
By construction, $\tilde g$ satisfies all the requirements of the theorem.
\end{proof}

\begin{proof}[Proof of Theorem \ref{thm: smoothing sec} (sectional-curvature case)] We use the same notation and setting as in the proof of the curvature-operator case. 

{\it Step 1. $C^1$-interpolation.} We employ the same interpolation polynomial and proceed to verify the sectional-curvature upper bound for $g_{\delta}$ on $\Gamma\times[-\delta,\delta]$.

{\it Substep 1.1. Tangential-normal sectional curvatures and mixed terms.} The argument for these two estimates is exactly the same as that in the curvature-operator case: it uses only the regularity across the corner and the negativity of $\secf^{+}+\secf^{-}$, and in particular does not invoke any curvature upper bound. We therefore omit the derivation and record only the resulting estimates. Namely, there exist positive constants $\lambda$ and $\Lambda$, independent of $\delta$, such that for all $t\in[-\delta,\delta]$ and for any vector $\beta$ and bivector $\psi$ tangent to $\Gamma_t$, we have
\begin{equation}\label{smoothupboundtn'}
 \Rm^{g_\delta}(\beta\wedge e_n,\beta\wedge e_n)\leq -\frac{\lambda}{\delta}|\beta|^2_{\gamma_\delta(t)},
\end{equation}
and
\begin{equation}\label{smoothupboundmixed'}
\left|\Rm^{g_\delta}(\psi,\beta\wedge e_n)\right|\leq \Lambda|\psi|_{\gamma_\delta(t)}|\beta|_{\gamma_\delta(t)}.
\end{equation}

{\it Substep 1.2. Tangential sectional curvatures.} Let $\varphi$ be an arbitrary non-zero decomposable bivector tangent to $\Gamma_t$, since we are in the sectional-curvature setting. By the assumption $\sec^g\leq K$, we have 
$$\Rm^g(\varphi,\varphi)\leq K|\varphi|^2_{\gamma(t)}.$$
As before, let
\[A_t=-\frac{\delta+t}{\delta}\secf^{+}+\frac{\delta-t}{\delta}\secf^{-}.\] Repeating the corresponding argument in the curvature-operator case, we obtain 
\begin{equation*}
\Rm^{g_\delta}(\varphi,\varphi)\leq \left(K+O(\delta)\right)|\varphi|^2_{\gamma_\delta(t)}+(\secf^-\wedge\secf^-)(\varphi,\varphi)-\frac{1}{4}(A_t\wedge A_t)(\varphi,\varphi).
\end{equation*}
It remains to prove
\begin{equation*}
(A_t\wedge A_t)(\varphi,\varphi)\geq 4(\secf^-\wedge\secf^-)(\varphi,\varphi).
\end{equation*}
Let $P$ be the $2$-plane determined by $\varphi$. On $P$, $-\left(\secf^{+}+\secf^{-}\right)$ induces an inner product and $\secf^{-}$ defines a self-adjoint operator with respect to this inner product. Let $\lambda_1$ and $\lambda_2$ be the eigenvalues of this self-adjoint operator, and let $v_1$ and $v_2$ be the corresponding unit eigenvectors. In particular, for $1\leq i,j\leq 2$, we have
\begin{equation*}
-\left(\secf^{+}+\secf^{-}\right)(v_i,v_j)=\delta_{ij}, \qquad \secf^{-}(v_i,v_j)=\lambda_i\delta_{ij}.
\end{equation*}
The $2$-convexity assumption implies that $\lambda_1+\lambda_2\geq 0$. Arguing as in the curvature-operator case, but now for the decomposable bivector $v_1\wedge v_2$, we obtain 
$$\left(A_t\wedge A_t\right)(v_1\wedge v_2,v_1\wedge v_2)\geq 4(\secf^-\wedge\secf^-)(v_1\wedge v_2,v_1\wedge v_2).$$
Since $\varphi=cv_1\wedge v_2$ for some constant $c$, the desired inequality follows. Finally, for $\delta>0$ sufficiently small,  uniformly in $\varphi$, we have
\begin{equation}\label{smoothupboundtt'}
\Rm^{g_\delta}(\varphi,\varphi)\leq \left(K+\frac{\varepsilon}{4}\right)|\varphi|^2_{\gamma_\delta(t)}.
\end{equation}

{\it Substep 1.3. General sectional curvatures.} Let $p\in\Gamma_t$, and let $P\subset T_pM$ be an arbitrary $2$-plane. Choose a basis $v, w$ of $P$. Decompose $v$ and $w$ as
$$
v=v_T+re_n,\qquad w=w_T+se_n,
$$
where $v_T$ and $w_T$ are tangent to $\Gamma_t$ and $s,r\in\mathbb R$. We consider the following two cases: 

{\it Case 1.} $r=0$ or $s=0$. These two subcases are essentially the same, so without loss of generality we assume $r=0$. Then
\begin{equation}\label{secexpansion}
\begin{split}
\Rm^{g_{\delta}}(v\wedge w,v\wedge w)
=&\,\Rm^{g_{\delta}}(v_T\wedge w_T,v_T\wedge w_T)+2s\Rm^{g_{\delta}}(v_T\wedge w_T,v_T\wedge e_n)\\&+s^2\Rm^{g_{\delta}}(v_T\wedge e_n,v_T\wedge e_n).
\end{split}
\end{equation}
Applying \eqref{smoothupboundtn'}, \eqref{smoothupboundmixed'}, and \eqref{smoothupboundtt'} to \eqref{secexpansion}, we arrive at
\begin{equation*}
\begin{split}
\Rm^{g_{\delta}}(v\wedge w,v\wedge w)\leq&\left(K+\frac{\varepsilon}{4}\right)|v_T\wedge w_T|^2_{\gamma_{\delta}(t)}+2\Lambda |s||v_T\wedge w_T|_{\gamma_{\delta}(t)}|v_T|_{\gamma_{\delta}(t)}\\
&\,-\frac{\lambda }{\delta}s^2|v_T|^2_{\gamma_{\delta}(t)}.
\end{split}
\end{equation*}
Using Young's inequality for the mixed term, we have 
$$2\Lambda |s||v_T\wedge w_T|_{\gamma_{\delta}(t)}|v_T|_{\gamma_{\delta}(t)}\leq\frac{2\Lambda^2\delta}{\lambda}\left|v_T\wedge w_T\right|^2_{\gamma_{\delta}(t)}+\frac{\lambda }{2\delta}s^2|v_T|^2_{\gamma_{\delta}(t)}.$$
Hence,
\begin{equation*}
\Rm^{g_{\delta}}(v\wedge w,v\wedge w)\leq \left(K+\frac{\varepsilon}{4}+\frac{2\Lambda^2\delta}{\lambda}\right)\left|v_T\wedge w_T\right|^2_{\gamma_{\delta}(t)}-\frac{\lambda }{2\delta}s^2|v_T|^2_{\gamma_{\delta}(t)}.
\end{equation*}
Note that 
\[|v\wedge w|^2_{g_{\delta}}=|v_T\wedge w_T|^2_{\gamma_{\delta}(t)}+s^2|v_T|^2_{\gamma_{\delta}(t)}.\] Thus, for $\delta>0$ sufficiently small, uniformly in $P$, we have
\[
\sec^{g_\delta}(P)\le K+\frac{\varepsilon}{2}.
\]

{\it Case 2.} $r\neq 0$, $s\neq 0$. Since $$\Span\{v_T+re_n,w_T+se_n\}=\Span\{sv_T-rw_T,w_T+se_n\},$$
this case reduces to Case 1. 

{\it Step 2. Standard mollification.}
We carry out the mollification exactly as in the curvature-operator case. To verify the sectional-curvature bound, we use the reduction in Substep 1.3, decomposing the curvature of a general 2-plane into tangential-normal, pure tangential, and mixed terms. We then estimate each term as in the curvature-operator case.
\end{proof}

Finally, we prove Proposition \ref{practical_use}, which is purely based on linear algebra.

\begin{proof}[Proof of Proposition \ref{practical_use}]
Fix $p\in \Gamma$ and a $2$-plane $P\subset T_p\Gamma$. Let \(\mu_1\) and \(\mu_2\) be the eigenvalues of \(\secf^{-}|_P\) with respect to \(-\secf^+|_P\). By the assumption \(-(\secf^-+\secf^+)|_P>0\), we have $\mu_i<1$ for $i=1,2$. By the $2$-convexity of $\secf^-$ with respect to $-\secf^+$, we have \(\mu_1 + \mu_2\geq 0\).

Now, let \(\lambda_1\) and \(\lambda_2\) be the eigenvalues of \(\secf^{-}|_P\) with respect to \(-(\secf^{+} + \secf^{-})|_P\). Then,  
\[
\lambda_i =\frac{\mu_i}{1-\mu_i},\qquad i=1,\,2.
\]
It follows that
\[
\lambda_1+\lambda_2=\frac{\mu_1+\mu_2-2\mu_1\mu_2}{(1-\mu_1)(1-\mu_2)}.
\]
Since \(0\leq\mu_1 + \mu_2<2\), the numerator of the right-hand side satisfies
\[
\mu_1+\mu_2-2\mu_1\mu_2\geq \frac{1}{2}(\mu_1+\mu_2)^2-2\mu_1\mu_2\geq 0.
\]
This concludes the proof.
\end{proof}

\section{Extension and obstruction results for positive curvature conditions}\label{Sec: 4}

\subsection{Extension for positive Ricci curvature} In this subsection, we prove the following result:
\begin{theorem}[Restatement of Theorem \ref{Thm:main 1}]\label{Thm: Ric extension}
Let \(M^n\) be a compact manifold with boundary, where \(n\geq2\). Given any \(K>0\), every smooth Riemannian metric \(\gamma\) on \(\partial M\) can be extended to a smooth Riemannian metric \(g\) on \(M\) such that \(\Ric^g>Kg\).
\end{theorem}

The proof uses a boundary-replacement argument. The following lemma guarantees the existence of a suitable initial metric.
\begin{lemma}\label{lem:large-boundary-metric}
Let \(M^n\) be a compact manifold with boundary, where
\(n\geq 2\). Given any \(K>0\) and any smooth Riemannian metric
\(\gamma\) on \(\partial M\), there exists a smooth Riemannian metric
\(g\) on \(M\) such that
\[
 \Ric^g>Kg \qquad\text{and}\qquad g|_{\partial M}>\gamma.
\]
\end{lemma}

The proof of this lemma combines the \(h\)-principle for sectional curvature on open manifolds, Nash--Kuiper stretching of the boundary embedding, and the isotopy extension theorem. Gromov gave a one-sentence sketch of this construction in the scalar-curvature setting \cite[Section~3.12.1, p.~228]{Gro2023}. Following his approach, we give a detailed proof in the Ricci curvature setting for the reader’s convenience.

\begin{proof}
We divide the proof into three steps.

\noindent\emph{Step 1: Construction of an ambient metric.}
Extend \(M\) slightly beyond \(\partial M\) to obtain an open manifold \(N\). By 
Theorem~1.1 in \cite{Streil2017}, \(N\) admits a smooth Riemannian metric \(\bar g\) such that
\(\Ric^{\bar g}>K\bar g\) on \(N\).

\noindent\emph{Step 2: Stretching the boundary embedding.}
Let \(Y=\partial M\), and let
\(
    F_0\colon Y\hookrightarrow N
\)
be the inclusion. Choose an open neighborhood \(\mathcal U\subset N\) of \(F_0(Y)\). We will construct a smooth isotopy $F_t: Y\hookrightarrow N$, starting at \(F_0\), such that
\[
    F_1^*\bar g>\gamma.
\]
Set \[
h_0=F_0^*\bar g, \qquad\text{and}\qquad q_t=h_0+t\gamma, \quad\text{for } t\in[0,1].
\]
Then $q_t$ is a strictly increasing path of metrics.

Let \(\operatorname{Emb}^1(Y,\mathcal U)\) denote the space of
\(C^1\) embeddings from \(Y\) into \(\mathcal U\), equipped with the
\(C^1\) topology, and let \(\operatorname{Met}^0(Y)\) denote the space
of continuous Riemannian metrics on \(Y\), equipped with the \(C^0\)
topology. Consider the pullback-metric map
\[
    P_{\bar g}\colon
    \operatorname{Emb}^1(Y,\mathcal U)
    \longrightarrow
    \operatorname{Met}^0(Y),
    \qquad
    P_{\bar g}(H)=H^*\bar g.
\]
By the increasing-path lifting clause of Gromov's \(C^1\)-fibration
theorem \cite[Section~2.4, pp.~186--187]{Gro2017}, the path
\(q_t\), together with its initial lift \(F_0\), admits a lift
\[
    G_t\in\operatorname{Emb}^1(Y,\mathcal U),
    \qquad
    G_0=F_0,
    \qquad
    G_t^*\bar g=q_t,
\]
which is continuous in \(t\) with respect to the \(C^1\) topology. In
particular,
\[
    G_1^*\bar g=h_0+\gamma>\gamma.
\]

Reparametrize the lift so that it is constant near \(t=0\). A standard relative smoothing argument, performed
uniformly in the fiberwise \(C^1\) topology, gives a jointly smooth
family
\[
    F\colon[0,1]\times Y\longrightarrow\mathcal U,
    \qquad F_t=F(t,\cdot),
\]
which agrees with \(F_0\) near \(t=0\) and is uniformly fiberwise
\(C^1\)-close to \(G_t\). Since \(Y\) is compact and
\(\operatorname{Emb}^1(Y,\mathcal U)\) is \(C^1\)-open, the
approximation may be chosen so that every \(F_t\) remains an
embedding. Moreover, the pullback-metric map is continuous from the
\(C^1\) topology to the \(C^0\) topology. Since
\[
    G_1^*\bar g-\gamma=h_0>0,
\]
the approximation may also be chosen so that
\(
    F_1^*\bar g>\gamma.
\)

\noindent\emph{Step 3: Isotopy extension.}
By the isotopy extension theorem
\cite[Chapter~8, Section~1, Theorem~1.3]{Hirsch}, the
isotopy \(F_t\) extends to an ambient diffeotopy
\[
    \Phi_t\colon N\longrightarrow N,
    \qquad
    \Phi_0=\operatorname{id}_N,
    \qquad
    \Phi_t\circ F_0=F_t.
\]
Define \(g=(\Phi_1|_M)^*\bar g.\) Then
\[
    \Ric^g
    =(\Phi_1|_M)^*\Ric^{\bar g}
    >K(\Phi_1|_M)^*\bar g
    =K g.
\]
On the boundary,
\[
    g|_{\partial M}
    =(\Phi_1\circ F_0)^*\bar g
    =F_1^*\bar g
    >\gamma.
\]
\end{proof}

We are now ready to prove Theorem \ref{Thm: Ric extension}.
\begin{proof}[Proof of Theorem \ref{Thm: Ric extension}]
By Lemma \ref{lem:large-boundary-metric}, we obtain a smooth metric $g_-$ on $M$ such that $\Ric^{g_-}>Kg_-$ and ${g_-}|_{\partial M}>\gamma$. Denote ${g_-}|_{\partial M}$ by $\gamma_0$. Let $\secf^-$ denote the second fundamental form of $\partial M$ in $(M,g_-)$ with respect to the inward unit normal.

Next, we construct the extension neck. Consider $N=\partial M\times [0,1]$ equipped with the base metric
$$g_{\rm{base}}=dt^2+\gamma(t),$$
where 
$$\gamma(t)=(1-t)\gamma_0+t\gamma.$$
Note that $\gamma(0)=\gamma_0$ and $\gamma(1)=\gamma$. Since $\gamma_0>\gamma$, $\gamma'(t)<0$ for all $t\in [0,1]$. By applying Lemma \ref{Lem:PRicciCcobordismlm}, we obtain a smooth metric $g_+$ on $N$ with $\Ric^{g_+}>Kg_+$ and satisfying the required boundary condition. Specifically, $g_+$ induces $\gamma_0$ on $\partial M\times\{0\}$ and $\gamma$ on $\partial M\times\{1\}$. In addition, the second fundamental form of $\partial M\times\{0\}$ in $(N,g_+)$ with respect to the inward unit normal, denoted by $\secf^+$, satisfies $\secf^++\secf^->0$.

Now, glue $M$ and $N$ together by identifying $\partial M\subset M$ with $\partial M\times\{0\}\subset N$. Denote the glued manifold $M\cup_{id} N$ by $\hat M$. On $\hat M$, consider the piecewise-defined metric $\hat g$, given by
\begin{equation*}
\hat g=\left\{
\begin{aligned}
&\,g_-\quad\mbox{on}\ M,\\
&\,g_+\quad\mbox{on}\ N.
\end{aligned}
\right.
\end{equation*}
Then $(\hat M,\hat g)$ is a Riemannian manifold with a corner along the gluing hypersurface. The second fundamental forms from the two sides are $\secf^+$ and $\secf^-$. Thus, $(\hat M,\hat g)$ satisfies the conditions of the \(K\)-lower-bound version of Theorem \ref{Lem: corner smoothing Ricci}, allowing us to construct a smooth metric $g$ on $\hat M$ with Ricci curvature $\Ric^g>Kg$ that agrees with $g_+$ near $\partial M\times\{1\}$. 

Since $\hat M$ is obtained by gluing a collar neighborhood of $\partial M$ to $M$, it is standard to construct a diffeomorphism 
\[
\Phi:(M,\partial M)\longrightarrow (\hat M,\partial M\times\{1\})
\]
such that the boundary map $\Phi|_{\partial M}$ agrees with the natural identification $\partial M\cong\partial M \times\{1\}$. It is straightforward to verify that the pullback metric $\Phi^* g$ satisfies all the required conditions.
\end{proof}

\subsection{Obstructions to stronger positive curvature conditions} In this subsection, we show that the metric extension problem for positive $k^{\mathrm{th}}$-intermediate Ricci curvature encounters local obstructions when $k\leq n-2$. 
\begin{theorem}[Restatement of Theorem \ref{Thm: obstruction Ric_k}]
Let $M$ be an $n$-dimensional compact manifold with boundary, and let $\sigma$ be a constant. Then the following hold:
\begin{itemize}[left=0.5cm]
\item[(i)]  For $n\geq 4$ and $k\leq n-3$, if $\gamma$ is a smooth metric on $\partial M$ with $\Ric_k^\gamma\leq \sigma$ at some point, then there is no smooth metric $g$ on any collar neighborhood of $\partial M$ such that $\Ric_k^g>\sigma$ and $g|_{\partial M}=\gamma$.
\item[(ii)] For $n\geq 3$, if $\gamma$ is a smooth metric on $\partial M$ with $\Ric^\gamma\leq \sigma$, and there exists a smooth metric $g$ on a collar neighborhood of $\partial M$ such that $\Ric_{n-2}^g>\sigma$ and $g|_{\partial M}=\gamma$, then the boundary tangent bundle $T(\partial M)$ splits as a direct sum of two nontrivial subbundles.
\end{itemize}
\end{theorem}
\begin{proof}[\it Proof of (i)] Let $p\in\partial M$ be the point where $\Ric^\gamma_k\leq \sigma$. Let $\{\lambda_i\}_{i=1}^{n-1}$ denote the principal curvatures of $\partial M$ at $p$ with respect to the inward normal, and let $\{v_i\}_{i=1}^{n-1}$ be the corresponding orthonormal eigenvectors. Without loss of generality, assume 
\begin{equation}\label{Eq: eigenvalue ordering}
\lambda_1\geq\lambda_2\geq\cdots \geq \lambda_{n-1}.
\end{equation}
By assumption, we have 
$$\Ric^\gamma_k(v_1;v_2,\ldots,v_{k+1})\leq \sigma,\quad\Ric_k^\gamma(v_{n-1};v_2,\ldots,v_{k+1})\leq \sigma,$$
and
$$\Ric^g_k(v_1;v_2,\ldots,v_{k+1})>\sigma,\quad\Ric_k^g(v_{n-1};v_2,\ldots,v_{k+1})>\sigma,$$
Subtracting these pairs and using the Gauss equation, we obtain
\begin{equation}\label{Eq: eigenvalue eq}
\lambda_1(\lambda_2+\cdots+\lambda_{k+1})<0\quad\mbox{and}\quad \lambda_{n-1}(\lambda_2+\cdots+\lambda_{k+1})<0.
\end{equation}
Thus, $\lambda_1$ and $\lambda_{n-1}$ must have the same sign. It follows from \eqref{Eq: eigenvalue ordering} that all $\lambda_i$ must have the same sign, which leads to a contradiction with \eqref{Eq: eigenvalue eq}.
\end{proof}
\begin{proof}[\it Proof of (ii)]
We continue to use $\{\lambda_i\}_{i=1}^{n-1}$ to denote the principal curvatures of $\partial M$ with respect to the inward normal, and $\{v_i\}_{i=1}^{n-1}$ to denote the corresponding orthonormal eigenvectors. By the Gauss equation and the assumption, for each $i$, we have the following inequality:
\begin{equation}\label{Eq: diff sign}
\lambda_i\sum_{j\neq i}\lambda_j=\Ric^\gamma(v_i,v_i)-\Ric^g_{n-2}(v_i; \hat v_i)<0,
\end{equation}
where $\hat v_i$ denotes the set $\{v_j\}_{j\neq i}^{n-1}$. This implies that all principal curvatures are nonzero. Since the principal curvatures are continuous functions on $\partial M$ and are nonzero, they must have definite signs. In addition, from inequality \eqref{Eq: diff sign}, we know that $\lambda_i$ cannot all have the same sign. 

Thus, the shape operator of $\partial M$, denoted by $S$, is an endomorphism of the tangent bundle $T(\partial M)$ with no zero eigenvalues. By transforming $S$ to meet the requirement of Lemma 2.8 in \cite{Hat17}, or by modifying the proof of Lemma 2.8 slightly, we conclude that $T(\partial M)$ has an $S$-invariant bundle decomposition: $$T(\partial M)=\xi_+\oplus \xi_-,$$ 
where $\xi_+$ and $\xi_-$ are nontrivial subbundles corresponding to positive and negative principal curvatures, respectively.
\end{proof}

As a corollary, we have:
\begin{corollary}[Restatement of Corollary \ref{Cor: even sphere}]
The following statements hold:
\begin{itemize}[left=0.5cm]
\item[(i)] For any compact manifold with boundary $M$ of dimension $n\geq 4$, there is a smooth metric $\gamma$ on $\partial M$ that cannot be extended to a smooth metric on $M$ with positive $(n-3)^{\rm th}$-intermediate Ricci curvature.
\item[(ii)] When $n\geq 5$ is odd, there is a smooth metric $\gamma$ on $\mathbb S^{n-1}$ that cannot be extended to a smooth metric on $\mathbb D^n$ with positive $(n-2)^{\rm th}$-intermediate Ricci curvature.
\end{itemize}
\end{corollary}
\begin{proof}
For (i), it suffices to construct a smooth metric $\gamma$ on $\partial M$ with negative sectional curvature at a point. This can be achieved by first constructing a hyperbolic metric $\gamma_1$ on a small neighborhood around this point and then extending $\gamma_1$ to a smooth metric $\gamma$ on the entire $\partial M$. 

For (ii), we need a smooth metric $\gamma$ on $\mathbb S^{n-1}$ with negative Ricci curvature. This is guaranteed by a general result due to Lohkamp \cite{Loh1994}, which states that any closed manifold of dimension at least $3$ admits a smooth metric with negative Ricci curvature. It remains to show that the tangent bundle of $\mathbb S^{n-1}$ is irreducible for odd $n$, which can be done by computing its Euler class. Since \(n\) is odd, \(e(T\mathbb S^{n-1})\neq0\). Suppose, by contradiction, that the tangent bundle of $\mathbb S^{n-1}$ is reducible. Then we have a non-trivial decomposition
$$T\mathbb S^{n-1}=\xi_1\oplus \xi_2,$$
where \(\xi_i\) is an \(n_i\)-plane bundle over \(\mathbb S^{n-1}\), with \(n_i\geq1\) and \(n_1+n_2=n-1\). Due to the facts $n_i\in\{1,\ldots, n-2\}$ and $H^k(\mathbb S^{n-1};\mathbb Z)=0$ for all $1\leq k\leq n-2$, we conclude 
$$e(\xi_i)=0 \in H^{n_i}(\mathbb S^{n-1};\mathbb Z).$$ Therefore we obtain 
$$e(T\mathbb S^{n-1})=e(\xi_1)\smile e(\xi_2)=0,$$ which leads to a contradiction.
\end{proof}

\section{Extension and obstruction results for negative curvature conditions}\label{Sec: 5}
\subsection{A global obstruction} In this subsection, we show that the metric extension problem for negative sectional curvature is not always solvable due to a global obstruction, in contrast to Theorem \ref{Thm: obstruction Ric_k} and Lemma \ref{Lem: NSecCcobordismlm}.

\begin{proposition}[Restatement of Proposition \ref{Prop: global obstruction}]\label{Prop: global obstruction'}
Let \(M=(\mathbb S^1\times\mathbb D^2)\#(\mathbb S^2\times\mathbb S^1)\),
where the connected sum is taken in the interior. Then no flat metric on
\(\partial M\cong\mathbb T^2\) can be extended to a smooth metric on \(M\)
with nonpositive sectional curvature.
\end{proposition}

To prove Proposition \ref{Prop: global obstruction'}, we use two classical results. The first characterizes upper curvature bounds in the Alexandrov sense for Riemannian manifolds with boundary. 

\begin{theorem} [Alexander--Berg--Bishop \cite{ABB}]\label{ACBoundary}
Let $(M,g)$ be a Riemannian manifold with boundary, and let \(K\in\mathbb R\). Then the following conditions are equivalent:
\begin{itemize}[left=0.5cm]
    \item [(i)] \((M,g)\) is locally \(\mathrm{CAT}(K)\).
    \item [(ii)] The sectional curvatures of $M$ and the outward sectional curvatures of $\partial M$ do not exceed $K$. Here, outward sectional curvatures of \(\partial M\) are its intrinsic sectional curvatures on two-planes on which the second fundamental form is negative definite.
\end{itemize}
\end{theorem}
When $K=0$, the above theorem was stated in \cite{Gro1978}. The second result is the generalized Cartan--Hadamard theorem, originating in Gromov’s work \cite{Gro1987}. The version stated below follows from \cite[Theorem II.4.1(2) and Corollary II.1.5]{BridsonHaefliger}.

\begin{theorem}[Generalized Cartan--Hadamard theorem] \label{gcartan-hadamard}
Let $X$ be a complete, connected length space that is locally \(\operatorname{CAT}(0)\). Then the universal cover \(\widetilde X\), equipped with the induced length metric, is a \(\operatorname{CAT}(0)\) space. Consequently, \(\widetilde X\) is contractible.  
\end{theorem}

\begin{proof}[Proof of Proposition \ref{Prop: global obstruction'}]
Suppose, to the contrary, that a flat metric $\gamma$ on $\partial M$ extends to a smooth metric $g$ on $M$ with nonpositive sectional curvature. Since $\gamma$ is flat, Theorem \ref{ACBoundary} shows that the intrinsic
length space $(M,d_g)$ is locally CAT$(0)$. The generalized
Cartan--Hadamard theorem therefore implies that the universal cover
$\widetilde M$ is a CAT$(0)$ space and hence contractible. It follows that
\(
\pi_2(M)\cong\pi_2(\widetilde M)=0.
\)
However, the $\mathbb S^2\times\mathbb S^1$ summand produces
a nontrivial element of $\pi_2(M)$, a contradiction. 
\end{proof}

\subsection{Solvability in $\mathcal C_n$ and $\mathcal C'_n$}
Despite the preceding obstruction, the metric extension problems for negative sectional curvature and negative curvature operator are always solvable for manifolds in $\mathcal C_n$ and $\mathcal C'_n$, respectively.

\begin{theorem}[Restatement of Theorem \ref{Thm:main 2}]\label{Thm: main 2'}
Let $M\in \mathcal C_n$ (respectively, $\mathcal C'_n$). Then any smooth metric $\gamma$ on $\partial M$ can be extended to a smooth metric on $M$ with negative sectional curvature (respectively, negative curvature operator) and strictly convex umbilical boundary.
\end{theorem}

The index-at-most-one condition in the definitions of $\mathcal C_n$ and $\mathcal C'_n$ does not directly give the $2$-convexity needed for gluing. After moving the boundary slightly outward, the following algebraic lemma produces a positive-definite tensor $B$ with respect to which the resulting boundary second fundamental form is strictly $2$-convex. We then prescribe the second fundamental form on the neck side to be a sufficiently large negative multiple of $B$.

\begin{lemma}\label{lem:two-convex-metric}
Let $\Sigma$ be a closed $m$-manifold, $m\geq 2$, and let $A$ be a smooth symmetric $2$-tensor on $\Sigma$. Suppose that $A$ has positive index of inertia at least $m-1$ at every point. Then there exists a smooth Riemannian metric $B$ on $\Sigma$ such that $A$ is strictly $2$-convex with respect to $B$.
\end{lemma}

 \begin{proof}[Proof of Lemma \ref{lem:two-convex-metric}]
We first construct the desired metric locally and then glue the local metrics using a partition of unity.

Fix $p\in\Sigma$. By the assumption on $A$, there exists a basis $\{e_1,\ldots,e_m\}$ of
$T_p\Sigma$ such that
\[
A_p=\diag(a_1,\ldots,a_m),
\qquad
0<a_2\leq\ldots\leq a_m.
\]
Choose $L>0$ sufficiently large so that $a_2>L^{-1}|a_1|$. With respect to the same basis, define an inner product $B_p$ on $T_p\Sigma$ by
\[
B_p=\diag(L,1,\ldots,1).
\]
The eigenvalues of \(A_p\) with respect to \(B_p\), listed in nondecreasing order, are
\[
\frac{a_1}{L},a_2,\ldots,a_m.
\]
By Ky Fan's minimum principle \cite{Fan},
\[
\frac{a_1}{L}+a_2
=
\min_{\substack{v_1,v_2\in T_p\Sigma\\
B_p(v_i,v_j)=\delta_{ij}}}
\left\{
A_p(v_1,v_1)+A_p(v_2,v_2)
\right\}.
\]
Consequently, \(A_p\) is strictly $2$-convex with respect to \(B_p\).

Extend $B_p$ smoothly to a Riemannian metric $B^{(p)}$ on a
neighborhood $U_p$ of $p$. After shrinking $U_p$ if necessary, $A$
is strictly $2$-convex with respect to $B^{(p)}$ on $U_p$. The neighborhoods $U_p$, $p\in\Sigma$, form an open cover of $\Sigma$. By compactness, select a finite subcover and relabel its members and
the corresponding metrics as $U_1,\ldots,U_N$ and $B^{(1)},\ldots,B^{(N)}$, respectively. Let $\{\rho_\alpha\}_{\alpha=1}^N$ be a partition of unity subordinate to this cover. Then
\[
B=\sum_{\alpha=1}^N\rho_\alpha B^{(\alpha)}
\]
is a smooth Riemannian metric. 

It remains to verify that $A$ is strictly $2$-convex with respect to $B$. Fix $x\in\Sigma$ and a $2$-plane $P=\Span\{X,Y\}\subset T_x\Sigma$. A direct calculation gives
\[
\tr_{B|_P}(A|_P)
=
\frac{
 A(X,X)B(Y,Y)+A(Y,Y)B(X,X)-2A(X,Y)B(X,Y)
}{
 B(X,X)B(Y,Y)-B(X,Y)^2
}.
\]
Note that the denominator is positive, while the numerator is linear
in \(B\) and hence is a partition-of-unity sum of the corresponding positive numerators for the $B^{(\alpha)}$. Therefore, $\tr_{B|_P}(A|_P)>0$. This proves the lemma.
\end{proof}

\begin{proof}[Proof of Theorem \ref{Thm: main 2'}]
We follow the same boundary-replacement strategy as in the proof of Theorem \ref{Thm: Ric extension}. We present only the negative sectional curvature case; the negative curvature operator case follows by the same argument.

By assumption, $M$ admits a smooth metric $g_0$ with negative sectional curvature such that the boundary has index at most $1$. Extend $g_0$ smoothly across $\partial M$, as is always possible; see
Pigola and Veronelli \cite[Theorem~A]{PV}. After restricting to a sufficiently small outer collar, the extended metric still has negative sectional curvature.

Using the extended metric, choose Gaussian normal coordinates
\(
(x,r)\in \partial M\times(-\varepsilon,\varepsilon)
\)
such that $r=0$ on $\partial M$, the portion of $M$ in this collar
corresponds to $r\leq 0$, and $\partial_r$ points outward from $M$. In this collar neighborhood, write
\[
g_0=dr^2+\eta_r.
\]
Let $\secf_r$ denote the second fundamental form of the $r$-slice with respect to the inward unit normal $-\partial_r$. Under our convention,
\[
\secf_r=\frac12\eta_r'.
\]
The Riccati equation gives
\[
\secf_r'=\secf_r^2-R(\partial_r,\cdot,\cdot,\partial_r).
\]
Thus, $\secf_r>\secf_0$ for every sufficiently small $r>0$. Since $\secf_0$
has negative index of inertia at most $1$, $\secf_r$ has positive index of
inertia at least $n-2$.

Fix such a small $r>0$ and identify the manifold obtained by adjoining the collar region $\partial M\times[0,r]$ to $M$ with $M$. Denote the resulting metric on $M$ by $g_-$.  After rescaling $g_-$ by a sufficiently small positive constant, we may assume that its induced boundary metric $\gamma_0$ satisfies
\[
\gamma>2\gamma_0.
\]
 Let
$\secf^-$ denote the second fundamental form of $\partial M$ in
$(M,g_-)$ with respect to the inward unit normal.

Applying Lemma \ref{lem:two-convex-metric} to $\secf^-$ on $\partial M$, we obtain a Riemannian
metric $B$ such that $\secf^-$ is strictly $2$-convex with respect to
$B$. Replacing $B$ by a sufficiently small positive multiple, which does not affect strict
$2$-convexity, we may assume that
\[
D:=\frac{1}{2}\gamma-\gamma_0-B>0.
\]

We first construct a neck whose outer boundary is strictly convex and
carries the metric $\frac{1}{2}\gamma$. Let
\(
N_1=\partial M\times[0,1]
\)
and consider the path
\[
\gamma_t^{(1)}=\gamma_0+tB+t^2D.
\]
Then $\{\gamma_t^{(1)}\}_{t\in [0,1]}$ satisfies the following:
\[
\gamma_0^{(1)}=\gamma_0,\qquad \gamma_1^{(1)}=\frac{1}{2}\gamma,\qquad \bigl(\gamma_t^{(1)}\bigr)'=B+2tD>0.
\]
Using the quasi-spherical construction in the proof of Lemma
\ref{Lem: NSecCcobordismlm}, we obtain a metric
\[
g_1=u_1(t)^2dt^2+\gamma_t^{(1)}
\]
on $N_1$ with negative sectional curvature. At $t=0$, the unit normal
pointing into $N_1$ is $u_1(0)^{-1}\partial_t$, and the corresponding
second fundamental form is
\[
\secf_0^{(1)}=-\frac{1}{2u_1(0)}B.
\]
The initial value $u_1(0)$ may be chosen sufficiently small so that
\[
\secf^-+\secf_0^{(1)}
=
\secf^--\frac{1}{2u_1(0)}B<0.
\]
Moreover, since $-\secf_0^{(1)}$ is a positive multiple of $B$, $\secf^-$ is strictly $2$-convex with respect to 
$-\secf_0^{(1)}$. 

Glue $(M,g_-)$ and $(N_1,g_1)$ together. By Theorem \ref{thm: smoothing sec} and Proposition \ref{practical_use}, we can smooth the resulting corner while preserving negative
sectional curvature and leaving the metric unchanged near $\partial M\times\{1\}$. At $\partial M\times\{1\}$, the inward unit
normal is $-u_1(1)^{-1}\partial_t$, and hence its second fundamental
form satisfies
\[
\secf_1^{(1)}
=
\frac{1}{2u_1(1)}
\bigl(\gamma_t^{(1)}\bigr)'\big|_{t=1}=\frac{1}{2u_1(1)}(B+2D)>0.
\]
Thus the new boundary is strictly convex and carries the metric $\frac{1}{2}\gamma$.

We next attach a second neck that realizes the prescribed boundary
metric and makes the final boundary umbilical. Let
\(
N_2=\partial M\times[0,1]
\)
and set
\[
\gamma_t^{(2)}=\frac{1+t}{2}\gamma.
\]
Then $\{\gamma_t^{(2)}\}_{t\in [0,1]}$ satisfies the following:
\[
\gamma_0^{(2)}=\frac{1}{2}\gamma,
\qquad
\gamma_1^{(2)}=\gamma,
\qquad
\bigl(\gamma_t^{(2)}\bigr)'=\frac{1}{2}\gamma>0.
\]
The same quasi-spherical construction gives a metric
\[
g_2=u_2(t)^2dt^2+\gamma_t^{(2)}
\]
on $N_2$ with negative sectional curvature. At $t=0$, its second
fundamental form with respect to the unit normal pointing into $N_2$
is
\[
\secf_0^{(2)}
=
-\frac{1}{2u_2(0)}
\bigl(\gamma_t^{(2)}\bigr)'\big|_{t=0}
=
-\frac{1}{4u_2(0)}\gamma.
\]
Attach $N_2$ to the manifold obtained above along its boundary carrying the metric $\frac{1}{2}\gamma$. By choosing $u_2(0)$ sufficiently small, we can arrange that
\[
\secf_1^{(1)}+\secf_0^{(2)}<0.
\]
Since $\secf_1^{(1)}>0$, Proposition \ref{practical_use} and
Theorem \ref{thm: smoothing sec} allow us to smooth the second
corner while preserving negative sectional curvature and leaving the
metric unchanged near $\partial M\times\{1\}\subset N_2$.

At the remaining boundary component $\partial M\times\{1\}\subset N_2$, the
inward unit normal is $-u_2(1)^{-1}\partial_t$, and its second
fundamental form is
\[
\secf_1^{(2)}
=
\frac{1}{2u_2(1)}
\bigl(\gamma_t^{(2)}\bigr)'\big|_{t=1}
=
\frac{1}{4u_2(1)}\gamma.
\]
Thus, the final boundary has induced metric $\gamma$ and is both
strictly convex and umbilical.

The resulting manifold $\hat M$ is obtained from $M$ by adjoining two collars
and is therefore diffeomorphic to $M$. Choose a diffeomorphism
\[
\Phi:(M,\partial M)\longrightarrow
(\hat M,\partial M\times\{1\})
\]
whose restriction to the boundary is the natural identification.
Then the pullback of the resulting metric by $\Phi$ satisfies all the
required conditions.
\end{proof}

\subsection{AH extensions with prescribed conformal infinity}\label{Prescribing conformal infinity}

In this subsection, we prove Theorem \ref{Thm:conformally compact}. We begin by recalling the notions of conformal compactness and asymptotic hyperbolicity.

\begin{definition}\label{CC}
Let \( M \) be a compact manifold with boundary. A \emph{defining function} for $\partial M$ is a smooth function \(\rho\) on \(M\) such that
\begin{equation*}
\left\{
\begin{aligned}
&\rho > 0 \ \,\quad \text{in} \  \mathring M, \\
&\rho= 0 \ \,\quad \text{on} \  \partial M, \\
&d\rho\neq 0 \quad \text{on} \  \partial M.
\end{aligned}
\right.
\end{equation*}
A Riemannian metric \(g\) on \(\mathring M\) is called
\emph{conformally compact} if there exists a defining function \(\rho\) such
that the rescaled metric \(\bar{g}=\rho^2g\) extends smoothly to a Riemannian
metric on \(M\), and hence induces a smooth Riemannian metric
\(\gamma\) on \(\partial M\).
If \(\tilde{\rho}\) is another defining function, then the boundary metrics
induced by \(\tilde{\rho}^{2}g\) and \(\rho^{2}g\) differ by a positive
conformal factor. Thus a conformally compact metric \(g\) induces a
well-defined conformal class \([\gamma]\) on \(\partial M\). We refer to
\((\partial M,[\gamma])\) as the \emph{conformal infinity} of
\((\mathring M,g)\).
\end{definition}

Given a defining function \(\rho\), a standard conformal-change calculation gives
\[
\Rm^g=-\left|d\rho\right|_{\bar g}^{2}\,g\wedge g+O(\rho)
\]
as \(\rho\to0\), where the error term is measured with respect to \(g\);
see \cite{Mazzeo}. Hence, if
\(\left|d\rho\right|_{\bar g}=1\) on \(\partial M\), then the sectional
curvatures of \(g\) tend to \(-1\) at infinity. This naturally motivates
the following definition.

\begin{definition}\label{WAH}
Let \(M\) be a compact manifold with boundary. A conformally compact metric
\(g\) on \(\mathring M\) is called \emph{asymptotically hyperbolic} if there
exists a defining function \(\rho\) such that
\(
\left|d\rho\right|_{\bar g}=1
\)
on $\partial M$.
\end{definition}

\begin{theorem}[Restatement of Theorem \ref{Thm:conformally compact}]\label{Thm:conformally compact'}
Let \(M\) be a compact manifold with boundary, and let \(g\) be a smooth metric on \(M\) with
negative sectional curvature such that the boundary has index at most \(1\)
everywhere. Then, for every smooth Riemannian metric \(\gamma\) on
\(\partial M\), there exist a complete asymptotically hyperbolic metric
\(\tilde g\) on \(\mathring M\) with negative sectional curvature and
conformal infinity \((\partial M,[\gamma])\), and an isometric embedding
\[
X\colon (M,g)\hookrightarrow (\mathring M,\tilde g).
\]
The same conclusion holds with negative sectional curvature replaced by negative curvature operator throughout.
\end{theorem}

\begin{proof}
We present only the negative sectional curvature case; the negative curvature operator case follows by the same argument. 

Recall the construction in the proof of Theorem \ref{Thm: main 2'}. We first extend \(g\) slightly beyond
\(\partial M\) and rescale the extended metric by a suitable positive constant.
We then attach a neck, whose outer boundary is
strictly convex and carries the metric \(\gamma\). Both the neck attachment and the corner smoothing can be carried out in the added region, leaving the rescaled metric unchanged on \(M\).
Rescaling the resulting metric back, we obtain a compact manifold \(M_-\),
diffeomorphic to \(M\), and a metric \(g_-\) of negative sectional curvature
on \(M_-\) such that the original \((M,g)\) is isometrically contained in
\((\mathring M_-,g_-)\), while \(\partial M_-\) is strictly convex and carries
the metric \(\lambda^2\gamma\) for some constant \(\lambda>0\). Denote the
resulting inclusion by
\[
X_0\colon (M,g)\hookrightarrow(\mathring M_-,g_-).
\]

Identify \(\partial M_-\) with \(\partial M\), and let \(\secf^-\) denote
its second fundamental form with respect to the inward unit normal. Since
\(\partial M\) is compact, there exists a constant \(C_1>0\) such that
\(
\secf^-\leq C_1\gamma.
\)

Set \(N=\partial M\times[0,\infty)\), and equip \(N\) with the warped-product metric
\[
g_+=dt^2+f(t)^2\gamma,
\]
where \(f\) will be chosen below. Let \(\secf^+\) denote the second
fundamental form of \(\partial M\times\{0\}\) with respect to the inward
unit normal. By Proposition \ref{Lem: curvature relations},
\begin{equation}\label{secfc}
\secf^+=-f(0)f'(0)\gamma.
\end{equation}

Let \(\{e_i\}_{i=1}^{n-1}\) be a local frame tangent to \(\partial M\),
and set \(e_n=\partial_t\). In what follows, the indices \(i,j,k,l\) range
from \(1\) to \(n-1\). Proposition \ref{Lem: curvature relations} gives
\begin{equation}\label{tnc}
R^{g_+}_{nijn}=-ff''\gamma_{ij},
\end{equation}
\begin{equation}\label{ttc}
R^{g_+}_{ijkl}
=
f^2R^\gamma_{ijkl}
-f^2f'^2
\left(
\gamma_{il}\gamma_{jk}-\gamma_{ik}\gamma_{jl}
\right),
\end{equation}
\begin{equation}\label{mixedc}
R^{g_+}_{ijkn}=0.
\end{equation}

Let \(\omega\) be an arbitrary bivector and decompose it as
\[
\omega=\psi+\beta\wedge e_n,
\]
where $\psi$ is a bivector and $\beta$ is a vector, both tangent to $\partial M$. Then
\begin{equation}\label{curoprexpansion2c}
\Rm^{g_+}\left(\omega,\omega\right)=\Rm^{g_+}\left(\psi,\psi\right)+2\Rm^{g_+}\left(\psi,\beta\wedge e_n\right)+\Rm^{g_+}\left(\beta\wedge e_n,\beta\wedge e_n\right).
\end{equation}
By compactness of \(\partial M\), there exists \(C_2>0\) such that
\[
\Rm^{\gamma}\leq C_2\gamma\wedge\gamma.
\]
Substituting \eqref{tnc}--\eqref{mixedc} into \eqref{curoprexpansion2c} yields
\[
\Rm^{g_+}(\omega,\omega)
\leq
f^2\left(C_2-f'^2\right)|\psi|_\gamma^2
-ff''|\beta|_\gamma^2.
\]

Choose \(a>0\) sufficiently large that
\[
(2a+\lambda)^2>C_2,
\qquad
\lambda(2a+\lambda)>C_1,
\]
and set
\begin{equation*}
f(t)=(a+\lambda)e^t-ae^{-t}.
\end{equation*}
Then
\[
f(0)=\lambda,\qquad
f'(0)=2a+\lambda,\qquad
f''=f>0.
\]
It follows that
\(
\Rm^{g_+}(\omega,\omega)<0
\)
for every nonzero bivector \(\omega\). Thus \(g_+\) has negative curvature
operator. Furthermore,
by \eqref{secfc},
\[
\secf^+
=
-\lambda(2a+\lambda)\gamma
<
-C_1\gamma,
\]
and hence
\(
\secf^-+\secf^+<0.
\)

Note \(\left.g_+\right|_{\partial M\times\{0\}}=\lambda^2\gamma=\left.g_-\right|_{\partial M_-}\). We may glue \(M_-\) and \(N\) along their boundaries. Let
\(
\hat M=M_-\cup_{\partial M}N
\)
and equip it with the resulting piecewise-smooth metric
\[
\hat g=
\begin{cases}
g_- & \text{on }M_-,\\
g_+ & \text{on }N.
\end{cases}
\]
Since \(\secf^->0\), \(\secf^+<0\), and \(\secf^-+\secf^+<0\), Theorem
\ref{thm: smoothing sec} and Proposition \ref{practical_use} allow us to smooth $\hat g$ to a metric \(\tilde g\) on \(\hat M\) with negative sectional curvature. The smoothing may be supported in an arbitrarily
small neighborhood of the gluing hypersurface so that \(\tilde g\) agrees with \(g_-\) on \(X_0(M)\) and with \(g_+\) on
\(\partial M\times[1,\infty)\).

It remains to examine the asymptotic geometry of \(\tilde g\). Compactify
\(\hat M\) by adjoining a copy of \(\partial M\) at \(t=\infty\), and
use
\(
\rho=e^{-t}
\)
as a boundary coordinate. Extend \(\rho\) smoothly and positively over the
compact part of \(\hat M\). On \(\partial M\times[1,\infty)\), we have
\[
\rho^2\tilde g
=
d\rho^2+
\left(a+\lambda-a\rho^2\right)^2\gamma.
\]
Thus \(\rho^2\tilde g\) extends smoothly across \(\rho=0\), and
\(
|d\rho|_{\rho^2\tilde g}^2=1.
\)
Moreover,
\[
\left.(\rho^2\tilde g)\right|_{\partial M}
=
(a+\lambda)^2\gamma.
\]
It follows that \(\tilde g\) is asymptotically hyperbolic with conformal infinity
\((\partial M,[\gamma])\). 

Finally, the resulting compactification is diffeomorphic to \(M\). Thus we may regard \(\tilde g\) as a metric on \(\mathring M\) and \(X_0\) as the required isometric embedding \(X\).
\end{proof}

\subsection{Structural properties of $\mathcal C_n$ and $\mathcal C'_n$}
\begin{theorem}\label{Thm: structure'}
The following statements hold:
\begin{itemize}[left=0.5cm]
\item[(i)] The classes $\mathcal C_n$ and $\mathcal C'_n$ are closed under boundary connected sums and attachments of boundary $1$-handles. In particular, $\mathcal C'_n$ contains all $n$-dimensional $1$-handlebodies. 
\item[(ii)] Let $N^k$ be a closed manifold that admits a metric with negative sectional curvature (respectively, negative curvature operator), and let $E$ be a rank $m$ vector bundle over $N$. Then the disk bundle \(D(E)\) belongs to \(\mathcal C_{k+m}\) (respectively, \(\mathcal C'_{k+m}\)).
\end{itemize}
\end{theorem}

The closure statements in part (i) will be proved by a Maskit-type gluing at conformal infinity. We first refine the preceding AH extension construction so that the metric is exactly hyperbolic near a prescribed point at infinity. Throughout this subsection, if \(x=(x_1,\ldots,x_k)\) denotes Euclidean
coordinates, we write
\[
dx^2:=\sum_{i=1}^k(dx_i)^2.
\]

\begin{lemma}[Refinement of Theorem \ref{Thm:conformally compact'}]\label{Lem: improve conformal}
In Theorem \ref{Thm:conformally compact'}, if $[\gamma]$ is locally conformally flat near $p\in \partial M$, then the AH metric $\tilde g$ may be chosen to be exactly hyperbolic near $p$. More precisely, there exist a neighborhood \(U\) of \(p\) in \(\partial M\), a constant \(\varepsilon>0\), and collar coordinates \((x,\rho)\) on \(U\times[0,\varepsilon)\), with \(\rho\) a defining function, such that
\begin{equation*}
\tilde g=\frac{1}{\rho^2}\left(d\rho^2+dx^2\right)\quad\text{on}\ U\times(0,\varepsilon).
\end{equation*}
\end{lemma}
\begin{proof}
After replacing \(\gamma\) by a conformally equivalent representative, we may assume that \(\gamma\) is flat in a neighborhood of \(p\). We repeat the proof of Theorem \ref{Thm:conformally compact'}, modifying only the warping function on the AH end.

Let \(\eta\colon[0,\infty)\to[0,1]\) be a smooth cut-off function such that
\(\eta\equiv1\) near \(0\) and \(\eta\equiv0\) on \([1,\infty)\). With \(a>0\) chosen as in the proof of Theorem \(\ref{Thm:conformally compact'}\), replace the warping function \(f\) on the AH end by
\[
f_R(t)=(a+\lambda)e^t-a\eta\left(\frac{t}{R}\right)e^{-t},
\]
where \(R>0\) will be chosen sufficiently large. We have
\[
f_R(t)
\geq
(a+\lambda)e^t-ae^{-t}
\geq
\lambda>0.
\]
Moreover, direct differentiation gives
\[
f_R''=f_R+O\big(R^{-1}\big)
\]
uniformly on \([0,\infty)\). Hence \(f_R''>0\) when \(R\) is sufficiently large. Consequently, 
\[
f_R'(t)\geq f_R'(0)=2a+\lambda.
\]
By the same curvature estimate as in the proof of Theorem \(\ref{Thm:conformally compact'}\), the modified warped end has negative curvature operator. Since \(\eta\) is constant near \(0\), $f_R$ agrees with the original warping function there. Hence the induced metric and second fundamental form at the gluing hypersurface are unchanged, and the corner-smoothing argument remains valid.

For \(t\geq R\), we have
\(
f_R(t)=(a+\lambda)e^t,
\)
and hence the resulting metric $\tilde g$ satisfies
\[
\tilde g=dt^2+(a+\lambda)^2e^{2t}\gamma
\]
in this region. Since \(\gamma\) is flat near \(p\), there exist a
neighborhood \(U\) of \(p\) and coordinates \(x\) on \(U\) such that
\(
(a+\lambda)^2\gamma=dx^2.
\)
Setting \(\rho=e^{-t}\), we obtain
\[
dt^2+(a+\lambda)^2e^{2t}\gamma
=
\frac{1}{\rho^2}\left(d\rho^2+dx^2\right)
\]
on \(U\times(0,e^{-R})\). Thus the AH extension is exactly hyperbolic near \(p\), as claimed.
\end{proof}

\begin{proof}[Proof of Theorem \ref{Thm: structure'}]
(i) We first prove closure under boundary connected sums. Let \(M_1,M_2\in\mathcal C_n\) (respectively, \(M_1,M_2\in\mathcal C'_n\)). For $i=1,2$, choose smooth metrics \(\gamma_i\) on \(\partial M_i\) that are flat near points \(p_i\in\partial M_i\). By Lemma \ref{Lem: improve conformal}, there exist an AH metric \(g_i\) on \(\mathring M_i\) and a defining
function \(\rho_i\) for \(\partial M_i\) such that \(g_i\) has negative
sectional curvature (respectively, negative curvature operator) and is
exactly hyperbolic near \(p_i\).

Consider the upper half-space model of the hyperbolic space
\[
\mathbb H^n=\{y\in\mathbb R^n\mid y_n>0\},
\quad
g_{\mathbb H}=\frac{dy^2}{y_n^2},
\]
and denote by \(B_s^+\) and \(S_s^+\) the Euclidean upper half-ball and upper
hemisphere of radius \(s\), respectively. There exist \(r>0\), compact neighborhoods \(V_i\) of
\(p_i\) in \(M_i\), and diffeomorphisms
\(
\Phi_i\colon V_i\longrightarrow\overline{B_{2r}^+}
\)
such that
\[
\rho_i=\Phi_i^*y_n\quad\text{on }V_i,
\qquad
g_i=\Phi_i^*g_{\mathbb H}
\quad\text{on }V_i\cap\mathring M_i.
\]

Excise a half-ball neighborhood of each \(p_i\) by setting
\[
\check M_i
:=
\overline{
M_i\setminus\Phi_i^{-1}\left(\overline{B_r^+}\right)
}.
\]
The excision introduces a hemispherical boundary hypersurface
\[
\Sigma_i
:=
\Phi_i^{-1}\left(\overline{S_r^+}\right)
\subset\partial\check M_i.
\]
Consider the inversion in the Euclidean sphere of radius \(r\),
\[
I_r\colon\mathbb R^n\setminus\{0\}
\longrightarrow
\mathbb R^n\setminus\{0\},
\qquad
I_r(y)=\frac{r^2}{|y|^2}y.
\]
It preserves \(\overline{\mathbb H^n}\), and its restriction to
\(\mathbb H^n\) is an isometry of
\((\mathbb H^n,g_{\mathbb H})\). Moreover, it fixes
\(\overline{S_r^+}\) pointwise and interchanges the regions
\(|y|<r\) and \(|y|>r\). Identify \(\Sigma_1\) and
\(\Sigma_2\) by
\[
\left.
\left(\Phi_2^{-1}\circ I_r\circ\Phi_1\right)
\right|_{\Sigma_1},
\]
and denote the resulting compact manifold by \(M_{\#}\). It is diffeomorphic
to the boundary connected sum
\(M_1\mathbin{\#_{\partial}}M_2\).

A neighborhood of the identified hypersurface in \(M_{\#}\) is parametrized
by the half-annulus
\[
A_r^+:=
\left\{
y\in\overline{\mathbb H^n}\,\middle|\,
\frac{r}{2}<|y|<2r
\right\},
\]
using \(\Phi_1\) on \(\check M_1\) and \(I_r\circ\Phi_2\) on
\(\check M_2\). Since \(I_r\) is a hyperbolic isometry, both \(g_1\) and \(g_2\) are
represented by \(g_{\mathbb H}\) in this chart. Hence they fit together
smoothly to define a metric \(g_{\#}\) on \(\mathring M_{\#}\) with
negative sectional curvature (respectively, negative curvature operator).

We next construct a defining function for \(g_{\#}\). Under the
half-annular parametrization above, the outer subannulus
\(A_r^+\cap\{|y|\geq r\}\)
comes from \(\check M_1\), where \(\rho_1\) is represented by \(y_n\).
The inner subannulus
\(A_r^+\cap\{|y|\leq r\}\)
comes from \(\check M_2\), where \(\rho_2\) is represented by
\[
I_r^*y_n
=
\frac{r^2}{|y|^2}y_n.
\]
Choose a smooth function
\(
q\colon\left(r/2,2r\right)\longrightarrow(0,\infty)
\)
such that \(q(s)=1\) near \(2r\) and \(q(s)=r^2/s^2\) near \(r/2\).
Set
\[
\rho_{\#}=q(|y|)y_n
\]
on \(A_r^+\), and let \(\rho_{\#}\) agree with \(\rho_i\) away from the
half-annular neighborhood. By the choice of \(q\), these definitions agree
on the corresponding overlaps and determine a smooth function
\(\rho_{\#}\) on \(M_{\#}\).

We now verify that \((\mathring M_{\#},g_{\#})\) is AH. Away from \(A_r^+\), the pair \((g_{\#},\rho_{\#})\) agrees with one of the original pairs \((g_i,\rho_i)\), so it suffices to verify the AH conditions on \(A_r^+\). Since
\(q>0\), $\rho_\#$ is positive for \(y_n>0\) and vanishes precisely when \(y_n=0\). Along \(y_n=0\),
\[
d\rho_{\#}=q(|y|)\,dy_n\neq 0.
\]
Thus \(\rho_{\#}\) is a defining function. On \(A_r^+\),
\[
\rho_{\#}^{\,2}g_{\#}=q(|y|)^2dy^2,
\]
which extends smoothly as a Riemannian metric up to \(y_n=0\). Along \(y_n=0\),
\[
|d\rho_{\#}|_{\rho_{\#}^{\,2}g_{\#}}=1.
\]
Therefore, \(g_{\#}\) is AH.

Since \(g_{\#}\) is AH, for sufficiently small \(\delta>0\), the compact domain
\(
\{\rho_{\#}\geq\delta\}
\)
is diffeomorphic to \(M_{\#}\) and has strictly convex boundary.
Therefore, 
\(
M_1 \mathbin{\#_{\partial}} M_2\in\mathcal C_n
\)
(respectively, 
\(
M_1 \mathbin{\#_{\partial}} M_2\in\mathcal C'_n).
\)

Repeating the argument using two distinct points on the conformal boundary of a single manifold \(M\) proves that both \(\mathcal C_n\) and \(\mathcal C'_n\) are closed under boundary \(1\)-handle attachments.

Finally, a compact geodesic ball in \(\mathbb H^n\) has negative curvature
operator and strictly convex boundary, and therefore belongs to
\(\mathcal C'_n\). Since every \(n\)-dimensional \(1\)-handlebody is obtained
from \(\mathbb D^n\) by successively attaching boundary \(1\)-handles, every
such handlebody belongs to \(\mathcal C'_n\).

(ii) The sectional curvature case also follows from Anderson’s construction \cite{Anderson}, since sufficiently small tubular neighborhoods of its totally geodesic zero section have strictly convex boundary. We give a direct local proof treating both curvature conditions.

Let $\pi:E\to N$ be the bundle projection, and let $h$ be a Riemannian metric on $N$ with negative sectional curvature
(respectively, negative curvature operator). Choose a fiber metric and a compatible connection on $E$, and let $g_0$ be the associated connection metric. Thus the horizontal and vertical distributions are orthogonal with respect to $g_0$,
with metrics $\pi^*h$ and the fiber metric, respectively. Write $r(v)=|v|$ for the fiber norm, and set $f=\lambda r^2$,
where $\lambda>0$ is a constant to be chosen later. Define
\[
g_\lambda=e^{2f}g_0.
\]
Since $r^2$ is smooth on $E$, so is $g_\lambda$. Fiber reflection
$v\mapsto-v$ is an isometry of both metrics, with fixed point
set equal to the zero section. Hence the zero section is
totally geodesic for both metrics, and both induce $h$ on it.

Along the zero section, we have 
\[
TE|_N\cong TN\oplus E\qquad\text{and}\qquad f=0,\ \  df=0.
\]
Moreover, $\nabla^2_{g_0}f$ vanishes on horizontal and mixed pairs and
equals $2\lambda g_0$ on vertical pairs. The conformal curvature
formula~\cite[Theorem~1.159(b)]{Besse} gives
\[
\Rm^{g_\lambda}
=e^{2f}\left[
\Rm^{g_0}-g_0\owedge
\left(\nabla^2_{g_0}f-df\otimes df
+\frac12|df|_{g_0}^2g_0\right)
\right],
\]
where $\owedge$ denotes the Kulkarni--Nomizu product. Hence, along the zero section,
\[
\Rm^{g_\lambda}
=\Rm^{g_0}-g_0\owedge\nabla^2_{g_0}f.
\]
Using the orthogonal decomposition
\[
\wedge^2TE|_N
=\wedge^2TN\oplus(TN\wedge E)\oplus\wedge^2E,
\]
write $\varphi=\alpha+\beta+\gamma$. Then
\begin{equation}\label{eq:disk-bundle-curvature}
\Rm^{g_\lambda}(\varphi,\varphi)
=\Rm^{g_0}(\varphi,\varphi)
-2\lambda|\beta|_{g_0}^2-4\lambda|\gamma|_{g_0}^2.
\end{equation}
In the sectional curvature case, we restrict to decomposable bivectors \(\varphi\), whose horizontal components \(\alpha\) are also decomposable.

Since the zero section is totally geodesic, the Gauss equation gives
\[
\left.\Rm^{g_0}\right|_{\wedge^2TN}=\Rm^h.
\]
By compactness of $N$ and the curvature assumption
on $h$, there exists $c>0$ such that
\[
\Rm^{g_0}(\alpha,\alpha)
=\Rm^h(\alpha,\alpha)
\le -2c|\alpha|_{g_0}^2.
\]
The remaining curvature coefficients of $g_0$ are uniformly
bounded along $N$. Applying Young's inequality to absorb the
cross terms, we obtain a constant $C>0$, independent of $\lambda$,
such that
\[
\Rm^{g_0}(\varphi,\varphi)
\le -c|\alpha|_{g_0}^2
+C\bigl(|\beta|_{g_0}^2+|\gamma|_{g_0}^2\bigr).
\]
Together with~\eqref{eq:disk-bundle-curvature}, this gives
\[
\Rm^{g_\lambda}(\varphi,\varphi)
\le -c|\alpha|_{g_0}^2
+(C-2\lambda)|\beta|_{g_0}^2
+(C-4\lambda)|\gamma|_{g_0}^2.
\]
Fix $\lambda>C/2$. Then $g_\lambda$ satisfies the desired
curvature condition along $N$. By compactness and continuity, the same curvature condition holds on
\(
D_\varepsilon(E)=\{r\leq\varepsilon\}
\)
for sufficiently small $\varepsilon>0$.

It remains to verify boundary convexity. For $g_0$, the inward unit normal to
$\partial D_\varepsilon(E)$ is $-\partial_r$. Its second fundamental form
$\secf_{g_0}$ vanishes on horizontal and mixed pairs and equals
$\varepsilon^{-1}g_0$ on vertical tangent vectors. In particular,
$\secf_{g_0}\geq0$. The conformal transformation
formula~\cite[(1.163)]{Besse} yields
\[
\secf_{g_\lambda}=e^{\lambda\varepsilon^2}\left(\secf_{g_0}+2\lambda\varepsilon\,g_0|_{T\partial D_\varepsilon(E)}\right)>0.
\]
Pulling back \(g_\lambda\) to \(D(E)\) by fiberwise dilation completes the proof.
\end{proof}

\subsection{Application to isometric realization}\label{realization}
In this subsection, we prove the following realization theorem:
\begin{theorem}[Restatement of Theorem \ref{Thm: realization}]\label{Thm: realization'}
Let \(M\in\mathcal C_n\) (respectively, \(\mathcal C'_n\)). For every smooth metric \(\gamma\) on \(\partial M\), there exist an asymptotically hyperbolic metric \(g\) on \(\mathring M\) with negative sectional curvature (respectively, negative curvature operator) and a smooth isometric embedding
\(
X\colon(\partial M,\gamma)\hookrightarrow(\mathring M,g).
\)
Moreover, the second fundamental form of \(X(\partial M)\) with respect to the inward normal is positive definite.
\end{theorem}

\begin{proof}
By Theorem \(\ref{Thm:main 2}\), there exists a smooth metric \(g_0\) on \(M\) with negative sectional curvature (respectively, negative curvature operator) such that \(g_0|_{\partial M}=\gamma\) and \(\partial M\) is strictly convex with respect to the inward unit normal. Applying Theorem \(\ref{Thm:conformally compact}\) to \((M,g_0)\), we obtain a complete AH metric \(\tilde g\) on \(\mathring M\) with negative sectional curvature (respectively, negative curvature operator) and an isometric embedding
\[
X\colon(M,g_0)\hookrightarrow(\mathring M,\tilde g).
\]
Then the restriction of \(X\) to \(\partial M\) is the required isometric embedding. 
\end{proof}

\begin{remark}\label{can'tsimconnec}
In general, one cannot expect the target Riemannian manifold to be simply connected. For example, let \((\Sigma,\gamma)\) be a closed orientable surface with nonpositive sectional curvature. Suppose that \((\Sigma,\gamma)\) admits an isometric embedding into a complete, simply connected Riemannian \(3\)-manifold \((M,g)\) with nonpositive sectional curvature. By the classical Cartan–Hadamard theorem, \(M\) is diffeomorphic to \(\mathbb R^3\). Identifying \(\Sigma\) with its image, let \(M_{\mathrm{ext}}\) denote the closure of the unbounded component of \(M\setminus\Sigma\). Both the interior and the boundary of \(M_{\mathrm{ext}}\) have nonpositive sectional curvature, and its intrinsic length metric is complete. Theorems 5.2 and 5.3 therefore imply that its universal cover is contractible, and hence \(\pi_2(M_{\mathrm{ext}})=0\). However, \(M_{\mathrm{ext}}\) retracts onto a sufficiently large sphere enclosing \(\Sigma\), so \(\pi_2(M_{\mathrm{ext}})\neq0\), a contradiction.
\end{remark}

\begin{theorem}[Restatement of Theorem \ref{Thm: sphererealization}]
For every smooth metric \(\gamma\) on \(\mathbb S^n\), there exist a complete metric \(g\) on \(\mathbb R^{n+1}\) with negative curvature operator and a smooth isometric embedding
\(
X\colon(\mathbb S^n,\gamma)\hookrightarrow(\mathbb R^{n+1},g).
\)
Moreover, \(g\) has constant negative sectional curvature outside a compact set, and the second fundamental form of \(X(\mathbb S^n)\) with respect to the inward normal is positive definite. 
\end{theorem}

\begin{proof}
By Theorem \(\ref{Thm:main 2}\), there exists a smooth metric \(g_0\) on \(\mathbb D^{n+1}\) with negative curvature operator such that \(g_0|_{\partial \mathbb D^{n+1}}=\gamma\) and \(\partial \mathbb D^{n+1}\) is strictly convex with respect to the inward unit normal. Let $\gamma_{\rm std}$ denote the standard metric on $\mathbb S^n$. Applying the construction in the proof of Theorem \ref{Thm:conformally compact'} to \((\mathbb D^{n+1},g_0)\), with \(\gamma_{\rm std}\) prescribed at conformal infinity, we obtain a complete metric \(\tilde g\) on \(\mathbb R^{n+1}\) with negative curvature operator and an isometric embedding
\[
X\colon(\mathbb D^{n+1},g_0)\hookrightarrow(\mathbb R^{n+1},\tilde g).
\]
Moreover, outside a compact set containing \(X(\mathbb D^{n+1})\), the manifold is identified with \(\mathbb S^n\times[0,\infty)\), and
\[
\tilde g=dt^2+f(t)^2\gamma_{\rm std},
\]
where $f>0$ and $f'>0$.

For \(\kappa>0\), choose \(r_\kappa>0\) such that
\[
\frac{1}{\kappa}\sinh(\kappa r_\kappa)=f(1).
\]
Truncate the warped-product end along \(\mathbb S^n\times\{1\}\) and attach the exterior of the geodesic ball of radius \(r_\kappa\) in the \((n+1)\)-dimensional hyperbolic space of curvature \(-\kappa^2\). The two boundaries have the same induced metric \(f(1)^2\gamma_{\rm std}\). With respect to the inward unit normals on the two sides, their second fundamental forms are
\[
f(1)f'(1)\gamma_{\rm std}
\qquad\text{and}\qquad
-f(1)\sqrt{1+\kappa^2f(1)^2}\gamma_{\rm std},
\]
respectively. For sufficiently large \(\kappa\), their sum is negative definite. 

Since the first form is positive definite, Theorem \ref{thm: smoothing sec} applies. Choosing the smoothing parameter sufficiently small, we obtain a smooth metric \(g\) with negative curvature operator on the resulting manifold, which is diffeomorphic to \(\mathbb R^{n+1}\). Moreover, $g$ agrees with the hyperbolic metric of curvature \(-\kappa^2\) outside a compact set. Since the smoothing takes place away from \(X(\mathbb D^{n+1})\), the restriction \(X|_{\partial\mathbb D^{n+1}}\) is the required isometric embedding, and its second fundamental form remains positive definite.
\end{proof}

\section*{Statements and Declarations}

\noindent\textbf{Funding.}
The first author was partially supported by the National Natural Science Foundation of China (Grant Nos. 12471054 and 12001292) and the Fundamental Research Funds for the Central Universities, Nankai University (Grant No. 050-63263075). The second author was partially supported by National Key R\&D Program of China 2023YFA1009900, NSFC grant 12401072, and the start-up fund from Westlake University.
\medskip

\noindent\textbf{Competing Interests.}
The authors declare that they have no competing interests.

\medskip

\noindent\textbf{Data Availability.}
Data sharing is not applicable to this article, as no datasets were generated or analyzed during the current study.

\end{document}